\documentclass{article}
\usepackage{amsmath, amsthm, amssymb, fullpage, xypic}
\usepackage[pdftex]{graphicx}
\usepackage[all, knot]{xy}
\usepackage{multirow}
\usepackage{hyperref}
\usepackage{cleveref}
  \hypersetup{colorlinks=true,citecolor=blue}

\xyoption{arc}
\xyoption{poly}
\xyoption{curve}
\xyoption{arrow}

\theoremstyle{definition}
\newtheorem{defn}{Definition}[section]

\theoremstyle{lemma}
\newtheorem{prop}[defn]{Proposition}

\theoremstyle{lemma}
\newtheorem{lemma}[defn]{Lemma}

\theoremstyle{lemma}
\newtheorem{thm}[defn]{Theorem}

\theoremstyle{lemma}
\newtheorem{cor}[defn]{Corollary}

\theoremstyle{lemma}

\theoremstyle{definition}
\newtheorem{rmk}[defn]{Remark}

\theoremstyle{definition}

\newcommand{\ra}{\rightarrow}

\newcommand{\lra}{\longrightarrow}

\newcommand{\sr}{\stackrel}
\newcommand{\hra} {\hookrightarrow}

\newcommand {\F} {\mathbb{F}}

\newcommand {\Z} {\mathbb{Z}}
\newcommand {\R} {\mathbb{R}}
\newcommand {\C} {\mathbb{C}}
\newcommand {\M} {\mathbb{M}}
\newcommand {\sus} {\Sigma}
\newcommand {\tensor} {\otimes}
\newcommand {\iso} {\cong}
\newcommand {\dsum} {\oplus}
\newcommand {\we} {\simeq}
\newcommand {\Ei} {E_{\infty}}
\newcommand {\Hi} {H_{\infty}}

\newcommand {\ii} {\mathrm{i}}

\newcommand {\uu} {\mathrm{u}}
\newcommand {\ff} {\mathrm{f}}

\newcommand {\oo} {\overline}
\newcommand {\wt} {\widetilde}
\newcommand {\ssra} {\Rightarrow}

\newcommand {\cV} {\mathcal{V}}
\newcommand {\cA} {\mathcal{A}}
\newcommand {\bu} {\bullet}
\newcommand {\mc} {\mathcal}

\newcommand {\eHF} {\mathrm{H}\underline{\mathbb{F}}}
\newcommand{\HFR}{\mathrm{H}\mathbb{F}^{\mathbb{R}}}
\newcommand {\eF} {\underline{\mathbb{F}}}

\newcommand {\HH} {\mathrm{H}}

\newcommand {\RP}{\mathbb{R}\mathrm{P}}
\newcommand {\tG}{\widetilde{\Gamma}}

\newcommand{\bss}{\begin{singlespace}}
\newcommand{\ess}{\end{singlespace}}
\newcommand{\Ab}{A_{\bullet}}
\newcommand{\Bb}{B_{\bullet}}
\newcommand{\Xb}{X_{\bullet}}
\newcommand{\Yb}{Y_{\bullet}}
\newcommand{\Gb}{\Gamma_{\bullet}}
\newcommand{\tGb}{\tGamma_{\bullet}}

\newcommand{\cra}{{\ooalign{$\longrightarrow$\cr\hidewidth$\circ$\hidewidth\cr}}}
\newcommand{\cla}{{\ooalign{$\longleftarrow$\cr\hidewidth$\circ$\hidewidth\cr}}}
\newcommand {\tGamma}{\widetilde{\Gamma}}

\DeclareMathOperator{\Ext}{\mathrm{Ext}}

\DeclareMathOperator{\colim}{\mathrm{colim}}
\DeclareMathOperator{\hocolim}{\mathrm{hocolim}}

\DeclareMathOperator{\Hom}{\mathrm{Hom}}

\title{Squaring operations in the $RO(C_2)$-graded and real motivic Adams spectral sequences}
\author{Sean Tilson}

\begin{document}
\maketitle

\begin{abstract}
	In this paper we establish a formula for computing $d_2(sq^i(x))$ where $x$ is a permanent cycle in the $C_2$-equivariant Adams spectral sequence or the motivic Adams spectral sequence over $Spec(\R)$.
	This requires establishing that the Adams towers have an $\Hi$-structure as well as determining the attaching maps for $C_2$-equivariant projective spaces.
	The attaching maps of $C_2$-equivariant projective spaces can then be used to determine the coefficients of differentials in both the equivariant and motivic case.
	At the end some sample computations are given.
\end{abstract}
\tableofcontents
\section{Introduction}
One of the early major advances in stable homotopy theory was the construction of the Adams spectral sequence.
It supplanted Serre's method as the main tool for computing stable homotopy groups of spheres.
Despite its advantage over Serre's algorithm, the issue of computing differentials remains a difficult one.
One approach to resolving this issue is to take more structure into account.
In the case of the Adams spectral sequence, this first began with work of Kahn in \cite{Kahn} and culminated in the work of Bruner, see Chapters 4, 5, and 6 of \cite{HRS}.
Intermediate and important work was done independently by Milgram and M\"akinen as well. 

The work of Bruner gives a formula for computing differentials on $P^i(x)$ where 
\[
x\in E_2^{s,t}(Y)=\Ext^{s,t}_{\mathcal{A}^{cl}_*}(\F_p,\HH_*(Y;\F_p)),
\]
$P^i$ is an algebraic Steenrod operation \`a la May (as in \cite{MaySteenrod}), and $Y$ is an $\Hi$-ring spectrum.
The Bruner formula was a significant advance for those working with the Adams spectral sequence, for example see \cite{BrunerInfiniteFamily} or \cite{BrunerNewDifferential}.
The computation of algebraic Steenrod operations on $\Ext$ has been partially automated by Bruner.
This reduces the computations of various differentials to an almost algorithmic task.
The input for this process is a ``nice'' representation of the element of $\Ext$ with which one is interested.
These formulae are not the same as requiring the differential to commute with operations.

If that were the case, the result would be 
\[
d_r(P^i(x))=0
\]
when $x$ is a permanent cycle.
These results can by interpreted as an identification of the error term that prevents such power operations from commuting with differentials.
In this paper, we prove the following analogue of a special case of Bruner's formula in the $C_2$-equivariant and real motivic settings.
The result is stated slightly ambiguously meaning we do not include superscripts to distinguish between the settings.

\begin{thm}
Consider the $C_2$-equivariant Adams spectral sequence converging to the $RO(C_2)$-graded equivariant stable homotopy groups of spheres or the real motivic Adams spectral sequence converging to the bigraded motivic stable homotopy groups of spheres.
For a permanent cycle $x \in \Ext^{s,p,q}_{\cA_*}(\M_2,\M_2)$, we have 
\[
d_2 sq^{i-1} x =
\alpha_{i,q} sq^i x
\]
where 
\[
\alpha_{i, q} = \left\{
\begin{array}{ll}
h_0  & i \equiv 1, q \equiv 1 \pmod{2} \\
h_0 + \rho h_1 & i \equiv 1, q \equiv 0 \pmod{2} \\
\rho h_1 & i \equiv 0, q \equiv 1 \pmod{2} \\
0 & i \equiv 0, q \equiv 0 \pmod{2} \\
\end{array}
\right.
\]
in $\Ext^{1,1,0}_{\cA_*}(\M_2,\M_2)$. 
\end{thm}

This $\Ext$ is the $E_2$-page of the Adams spectral sequence.
In the $C_2$-equivariant case it is based on equivariant cohomology with coefficients in the ``constant'' Mackey functor taking value $\F_2$.
In the real motivic case the spectral sequence is based on ordinary motivic cohomology with coefficients in $\F_2$.
These $\Ext$-groups are computed in the category of bigraded comodules over the equivariant or motivic dual Steenrod algebra $\cA_*$, or equivalently modules over the equivariant or motivic Steenrod algebra $\cA$.
$\M_2$ denotes the coefficient group of equivariant or motivic (co)homology, of which $\rho$ is an element.
See \Cref{subsec:recollections equivariant coh} for equivariant details and \Cref{subsec: motivic algebraic condition} for motivic details.
Here, $s$ is the Adams filtration, $p$ is the topological dimension, and $q$ is the so called weight which we have in both the $RO(C_2)$-graded and motivic settings.
The elements $h_0$ and $h_1$ come from the elements $sq^1$ and $sq^2$ in the Steenrod algebra.
They correspond to the elements $\tau_0$ and $\xi_1$ in the dual Steenrod algebra, respectively.
These four different cases come from the variation in the attaching map of cells in projective space which depend on $i$ and $q$.

The following is a simple application of these results.
It does rely on computations of the products structure in the relevant $\Ext$-groups, see \cite{DuggerIsaksenLowMWstems}.
\begin{cor}
In the $RO(C_2)$-graded and real motivic Adams spectral sequences we have that
\begin{center}
$d_2(h_2)=h_0h_1^2=0$, $d_2(h_3)=(h_0+\rho h_1)h_2^2=0$, and $d_2(h_4)=(h_0+\rho h_1)h_3^2=h_0 h_3^2\neq=0$.
\end{center}
\end{cor}

This is explained in \Cref{subsec: apps} along with implications in the real motivic Adams spectral sequence.
These sample results are not new, but this proof is.
With greater knowledge of the $C_2$-equivariant $\Ext$ group above even more can be determined.
The results in that subsection only rely on the main result and algebraic Steenrod operations as discussed in \Cref{sec:power ops}.
The interested reader is encouraged to skip to this section.

The main insight of this approach of Bruner is to recognize that $Y$ being an $\Hi$-ring spectrum implies that the Adams filtration is an $\Hi$-filtration, see \Cref{defn:Hinfty filtration} and \Cref{thm: hinfty}.
This implies that the $E_1$-page of the spectral sequence has the necessary structure so that its homology, the $E_2$-page, has power operations as outlined by May in \cite{MaySteenrod} (for more recent exposition see the discussion in chapter 4 of \cite{BrunerAdamsPrimer}).
In particular, we do not feel that these methods or results are specific to the Adams spectral sequence or $\Hi$-ring spectra.
These types of results should be available for other spectral sequences as well, as is evidenced by the work of Bruner and Rognes on the Homotopy Fixed point spectral sequence in \cite{BrunerRognes}.
We believe other cocontinuous monads should be considered also. 
We hope that our exposition here will allow an adaptation of the method to other spectral sequences.

\begin{rmk}
Given that we are working in the equivariant and motivic settings we must say a word about which notions of $\Ei$-ring spectra, and hence $\Hi$-ring spectra, we consider.
We focus on naive $\Ei$-ring spectra as opposed to the theories involving norms.
Equivariantly, these are the operads where the constituents spaces have a trivial $C_2$-action.
Motivically, these are operads that are completely simplicial, that is they are constant presheaves.
Our reason for this is motivated by the structure present on the $E_2$-page of the Adams spectral sequences.
Whichever choice we make, we will look at filtered algebras over the associated monad.
Such a filtered algebra will have an induced structure on the associated graded complex.
This structure will induce Steenrod operations \`a la May on the $E_2$-page.
The equivariant Adams spectral sequence we consider is that of Hu and Kriz, \cite{HuKriz}, which has an $E_2$-page that has an extra grading, but no extra symmetry.
The motivic spectral sequence we consider is that of Dugger and Isaksen which converges to bigraded stable stems as opposed to bigraded homotopy sheaves.
This lack of extra symmetry implies that we should consider the naive $\Ei$-ring spectra.

The richer notions of commutative ring spectra in the equivariant and motivic setting should be considered, perhaps when the input is an $\varepsilon$-commutative ring in the sense of Lemma 2.12 and 2.17 of Hu and Kriz, see \cite{HuKriz}.
We believe that they will be relevant for Mackey functor Adams spectral sequences, perhaps with less trivial coefficients.
\end{rmk}

\subsection{Outline of the paper}
We have relegated our motivic considerations to a single section before our main theorem.
Our reason for doing this is that the equivariant and motivic arguments are very similar while the equivariant setting may be more comfortable or familiar.
Once this material has been digested then it only requires a few references or comments before we can assemble our main theorem in both cases.
We have tried to employ subscripts or superscripts to distinguish between the equivariant and real motivic cases when ambiguity is possible.
However, motivic spectra don't appear before \Cref{sec:motivic} and so objects before hand can be assumed to be $C_2$-equivariant.
The material in \Cref{sec:filt} and \Cref{sec:outline} is not specific to either case.

First, we recall basic facts about equivariant stable homotopy theory.
These include definitions regarding the notion of cell complex as well as relevant facts about the dual $C_2$-equivariant Steenrod algebra $\cA^{C_2}_*$ and  the $RO(C_2)$-graded Adams spectral sequence of Hu and Kriz.
Then in \Cref{sec:filt} we recall basic facts about filtrations and multiplicative structures.
These will be the basic objects that we need to manipulate throughout.
In \Cref{sec:outline}, we outline the general method of the proof and isolate the key criteria that allow for the result.
This will be useful for understanding the paper and providing similar results in other settings.
We show in \Cref{sec:power ops} that the $E_2$-page carries the necessary structure for Bruner's proof of Lemma 2.3 in Chapter 4 of \cite{HRS} to hold.
This implies that the ``big'' Steenrod algebra, where $sq^0$ does not necessarily act as the identity, acts on our $E_2$-page.
\Cref{sec:Hinfty} contains the proof that the Adams tower we consider does indeed possess an $\Hi$-structure as described in \Cref{defn:Hinfty filtration}.
The proof uses obstruction theory for stable model categories as developed in \cite{ChristensenDwyerIsaksenObstModelCats}.
The following two sections deal with the more geometric parts of the proof.
This involves identifying extended powers of representation spheres with truncated equivariant projective spaces in \Cref{sec:Thom}, which we know no way of doing motivically.
We then compute the attaching maps of said projective spaces in \Cref{sec:rppqs}.
This will be all that is necessary to determine the coefficients of the differentials.
With these two final ingredients we are able to assemble our result \Cref{thm:diff} in the final section.
Before this last section we recall necessary motivic results in \Cref{sec:motivic} as well as make the necessary changes to the previous results so that they apply motivically.
In particular, there is motivic analogue of \Cref{sec:Thom} and so instead we use the $C_2$-equivariant Betti realization functor of Heller and Ormsby.
In \Cref{subsec: apps}, we also give some brief sample computations.

Similar equivariant results were obtained in my thesis.
There the results were framed so as to apply to a general $C_2$-equivariant Adams spectral sequence satisfying some hypotheses, see \Cref{hypotheses}.
Here we establish that the $RO(C_2)$-graded Adams spectral sequence constructed by Hu and Kriz in \cite{HuKriz} does indeed satisfy these hypotheses.

\begin{rmk}
Given that we frequently go back and forth between the equivariant and motivic situation, a few words are in order about notation.
We have tried to not overburden the notation with superscripts and subscripts.
The majority of the document is equivariant in nature, by which we mean undecorated objects should be assumed to be $C_2$-equivariant.
There are exceptions to this.
Notably, \Cref{sec:filt} and \Cref{sec:outline} don't depend on a particular context and are valid in both.
In both sections, the $\Hi$-structures are naive in the same sense.
\Cref{sec:motivic} is obviously about motivic phenomena and thus all objects therein should be interpreted as such except when decorated by $C_2$'s.
\Cref{subsec:proof} is written so that the majority of the arguments are independent of context.
This is until we must apply the results of \Cref{sec:rppqs} and \Cref{subsec: motivic extended powers of spheres}, and then we decorate our spheres so as to indicate the relevant context.
In summary, undecorated spectra should be interpreted $C_2$-equivariant spectra.
\end{rmk}

\section*{Acknowledgments}
The equivariant parts of this document contains work I did for my thesis at Wayne State University under the guidance of Robert Bruner.
I would especially like to thank him and Daniel Isaksen for their guidance along with the other members of my committee.
I would like to thank Jeremiah Heller for very helpful real motivic conversations.
I must also thank Tobias Barthel, Jens Hornbostel, Irakli Patchkoria, Eric Peterson, and Dylan Wilson.

\section{Background on equivariant homotopy theory}
This section is a very brief refresher on necessary equivariant facts.
More details can be found in \cite{HHRKervaire}, \cite{HuKriz}, \cite{SchwedeEquivariant}, \cite{Alaska}, and \cite{DwilsonEquivDL}. 
We review some basics about equivariant cell complexes and then recall a few results about ordinary equivariant cohomology.

\subsection{$C_2$-equivariant CW-complexes}
\label{subsec:CWcplexes}
We will make heavy use of this material in \Cref{sec:rppqs}.
We denote the trivial representation by $\R^{1,0}$ and the sign representation by $\R^{1,1}$.
The direct sum $(p-q)\R^{1,0} \dsum (q)\R^{1,1}$ will be denoted by $\R^{p,q}$ where $C_2$ acts by negating the last $q$ coordinates.
Two useful maps to the category of $\R$-vector spaces from the category of $C_2$-representations are the forgetful functor denoted by $\uu$ and the fixed point functor denoted by $\ff$.
Clearly, we have that $\uu(\R^{p,q})=\R^p$ and $\ff(\R^{p,q})=\R^{p-q}$.
These functors extend to functors from equivariant spaces to spaces.
The extension of $\uu$ and $\ff$ preserve cofiber sequences of finite type $C_2$-cell complexes, see Chapters I, IX, and X of \cite{Alaska}.
When we extend these fixed point functors to spectra we will be working with geometric fixed points.
While most of our computations in \Cref{sec:rppqs} take place in the unstable category of $C_2$-cell complexes, the maps we want in the end are elements of the stable homotopy groups of spheres.
Thus we use geometric fixed points $\Phi^G(-)$ which have the property that
\[
\Phi^G(\sus^{\infty}X)\we \sus^{\infty}(X^G).
\]
for any $G$-cell complex $X$.

We will build cell complexes from representation cells, as discussed in the following section.
In particular, we use $S^{p,q}$ to denote $\R^{p,q} \cup \{ \infty \}$.
Thus we have that $\uu(S^{p,q})=S^p$ and $\ff(S^{p,q})=S^{p-q}$. 

\subsubsection{Classical $G$-CW complexes vs. $Rep(G)$-cell complexes}
A $G$-CW-complexes $X$ is built by attaching cells of the form $G/H_+ \wedge D^n$ to $X^n$ along attaching maps 
\[
G/H_+ \wedge S^n \to X^n.
\]
These are the main objects of study in $G$-equivariant homotopy theory.
However, we will be working with spaces built out of representation cells, these are referred to as $Rep(G)$-cell complexes in the literature.
$Rep(G)$-cell complexes are built out of cells of the form $S(V)=\{v\in V\vert \left\langle v,v \right\rangle=1 \}$ and $D(V)=\{v\in V\vert \left\langle v,v \right\rangle \leq 1 \}$ for $V$ a $G$ representation with $G$-invariant inner product.
Our $S^{p,q}$'s are such spaces as they are equivariantly homeomorphic to $D(\R^{p,q})/S(\R^{p,q})$.
Also, if $p-1\geq q$, then $\partial D(\R^{p,q})$ as $S^{p-1,q}$.

It turns out that each $Rep(C_2)$-cell complex is a $C_2$-CW complex, and vice versa.
As $Rep(C_2)$-cell complexes are built from cofiber sequences of $S^{p,q}$'s, it will be sufficient to show that each of these is a $C_2$-CW complex.
Note that $S^0\iso (C_2/C_2)_+$ as a pointed $C_2$-CW complex.
The underlying space of the cofiber of 
\[
C_{2+} \lra S^0
\]
is $S^1$ with $C_2$-action equivalent to complex conjugation restricted to the unit circle in $\mathbb{C}$.
This is precisely the $S^{1,1}$ we have described above.
Similarly, $S^{1,0}$ is constructed by attaching $(C_2/C_2)_+ \wedge D^1$ to a single point $+$ along the collapse map 
\[
S^0 \lra +.
\]
These same cofiber sequences show that the $C_2$-CW cells $C_{2+}\wedge S^{n,0}$ and $C_2/C_{2+}\wedge S^{n,0}$ are also $Rep(C_2)$-cell complexes.
Each $S^{p,q}$ is the smash product $(S^{1,0})^{\wedge p-q} \wedge (S^{1,1})^{\wedge q}$.
Since $C_2$-CW complexes are closed under smash products we have that each $S^{p,q}$ is a $C_2$-CW complex.

This is particular to the case of $G=C_2$.
For general finite $G$ the issue is more subtle.
In fact, this is not true for other cyclic groups.
Our motivation for working with $Rep(C_2)$-cell complexes is that equivariant stable homotopy groups are computed as maps out of representation spheres.
When we investigate extended powers of representation spheres the resulting objects will be $Rep(C_2)$-cell complexes.
The above argument also implies that a map inducing an isomorphism in $RO(C_2)$-graded homotopy groups is a weak equivalence.

\subsection{Recollections on ``ordinary'' Equivariant Cohomology}
\label{subsec:recollections equivariant coh}
Our Adams spectral sequence is based on ordinary equivariant cohomology with coefficients in the ``constant'' Mackey functor $\eF_2$ (the value of the functor is constant but the maps are not all the identity).
Our invariants also take values in $RO(C_2)$-graded abelian groups, which are just bigraded groups.
Our grading convention is dictated by the notation in the previous subsection, and therefore will maximize compatibility with the Betti realization map from motivic homotopy theory over $\R$.

We then have the following results regarding ordinary equivariant homology.

\begin{prop} [Hu-Kriz]
\label{prop:Hu Kriz}
The equivariant homology of a point is given by
\[
\M_2:=\pi_{**}\eHF_2\iso \F_2 [\tau,\rho]\dsum \F_2 \{\frac{\theta}{\tau^i\rho^j}\}
\]
for every $i,j \in \Z_{\geq 0}$.
The elements $\tau$, $\rho$, and $\theta$ are in bidegrees $(-1,0)$, $(-1,-1)$, and $(2,0)$ respectively.
\end{prop}
Note that the class $\theta$ is infinitely $\tau$ and $\rho$ divisible as well as $\tau$ and $\rho$ torsion.
This ``negative'' cone is a complication not found in the motivic setting and gives rise to algebraic difficulties.
Our notation differs from theirs, our $\rho$ and $\tau$ correspond to their $a$ and $\sigma^{-1}$ respectively.
The class $\theta$ is not given a name by Hu and Kriz but it comes from their spectrum $H^f$, which is their notation for $E\Z_{2+}\wedge \eHF_2$.

\begin{thm} [Hu-Kriz]
\label{thm:flatness}
The equivariant dual Steenrod algebra is 
\[
\cA^{C_2}_*:=\pi_{**}\eHF_{2}\wedge\eHF_2 \iso \frac{\M_2[\xi_i,\tau_i,\alpha]}{\tau_0\rho=\alpha+\tau,\tau_i^2=\tau_{i+1}\rho+\xi_{i+1}\alpha}
\]
where $i\in \Z_{\geq0}$.
The dual Steenrod $\cA^{C_2}_*$ is flat and even free over $\M_2$.
\end{thm}
This is proved as Theorem 6.41 and Corollary 6.45 of \cite{HuKriz}.
The class $\alpha$ is given by the right unit applied to our class $\tau$, Hu and Kriz denote the class $\rho^{-1}$.
The bidegree of $\xi_i$ is $(2^{i+1}-2,2^i-1)$ and the bidegree of $\tau_i$ is $(2^{i+1}-1,2^i-1)$
Hu and Kriz also describe the rest of the Hopf algebroid structure, but that won't be relevant here.
The flatness is crucial though, as it provides for the following corollary.

\begin{cor}
For a finite $C_2$-spectrum $X$ there is a convergent Adams spectral sequence
\[
\Ext_{\cA^{C_2}_*}(\M_2 ,\mathrm{H}_{**}(X;\eF_2) ) \ssra (\pi_{**}(X))^{ \wedge }_2.
\]
\end{cor}

This spectral sequence is constructed in the standard way.
The convergence properties can be derived using the classical proof of convergence as a guide.
The freeness also allows us to directly apply the work of Bruner to construct an action of algebraic Steenrod operations on the $E_2$-page of this spectral sequence.
For more details, see \Cref{sec:power ops}.

We will use a few other results about equivariant cohomology.
Kronholm establishes in \cite{KronholmFree} that the (co)homology of a finite $Rep(C_2)$-cell complex is always free as a module over the (co)homology of a point.
This paired with the Universal Coefficient and K\"unneth spectral sequences of Lewis and Mandell developed in \cite{LewisMandellUCandKSS} will allow us to prove \Cref{prop: splitting}.
We will refer to these as the $RO(C_2)$-graded UCSS and $RO(C_2)$-graded KSS respectively, this is in spite of the fact that we are indeed referring to the full $RO(C_2)$-graded Mackey functor version of the spectral sequence.
However, all applications will be applied to free $RO(C_2)$-graded $\M_2$-modules which in turn are end up being free $RO(C_2)$-graded $\M_2$-modules in Mackey functors
This works by noting that if an $RO(C_2)$-graded abelian group is free as an $\M_2$-module then it is also free as an $\M_2$-module object in $RO(C_2)$-graded Mackey functors.
Thus the spectral sequences collapse and are determined by the underlying $RO(C_2)$-graded $\M_2$-module as we are always working in the category of $\M_2$-modules.

\begin{lemma}
\label{lemma:equi splitting}
If $X$ is a $C_2$-equivariant cellular spectrum whose $RO(C_2)$-graded homology is free over $\M_2$ then $X\wedge \eHF_2$ splits as a wedge of suspensions of $\eHF_2$'s.
In particular, it's homology is free as an $RO(C_2)$-graded $\M_2$-module in Mackey functors.
\end{lemma}

That the homology is free gives produces a map
\[
\bigvee S^{\alpha}\wedge \eHF_2 \lra X \wedge \eHF_2
\]
which will induce an isomorphism after applying $\pi_{**}$.
Since $X$ is a cellular spectrum such a map is a weak equivalence.
This follows from the fact that stably the class of orbit and representation spheres coincide, which is specific to the group $C_2$.
We learned this from Justin Noel's answer \cite{NoelMOanswer}.

Thus many equivariant spectra split when smashed with $\eHF_2$.
We will use this explicitly in order to construct the $\Hi$-structure of the Adams filtration in \Cref{thm: hinfty}.

\section{Background on filtered spectra, monoidal structures}
\label{sec:filt}
Here we recall a few constructions with filtered spectra that will be necessary.
These definitions are necessary for the outline of the method in the next section.
We work in the category of $C_2$-equivariant Orthogonal spectra as described in \cite{HHRKervaire}.
However, these definitions are not particular to the model of (motivic, equivariant, or classical) spectra being used.
Much of these definitions are inspired by the work of Bruner in \cite{HRS}, their seeds can be found there.

\begin{defn}
A filtered spectrum or filtration is a sequence of cofibrations
\[
\cdots \hra Y_{i-1} \hra Y_i \hra Y_{i+1} \hra \cdots.
\]
We denote the single filtered object as $\Yb$.
The associated graded complex of a filtered spectrum $\Yb$ is the complex of spectra
\[
\cdots \cla Y_{i-1}/Y_{i-2} \cla Y_i/Y_{i-1} \cla Y_{i+1}/Y_i \cla \cdots
\]
denoted by $E^0(\Yb)$.
\end{defn}

We use 
\[
A \cra B
\]
as a short hand for a map that changes degree, for example
\[
A \lra \sus B.
\]
By $Y_i/Y_{i-1}$ we mean the cofiber of the cofibration
\[
Y_{i-1} \hra Y_i
\]
that is part of the data of the filtered spectrum $\Yb$.
The maps on the associated graded complex
\[
Y_{i+1}/Y_i \cra Y_i/Y_{i-1}
\]
are defined as the composites
\[
Y_{i+1}/Y_i \ra \sus Y_i \ra \sus Y_i/Y_{i-1}.
\]
We will also call a filtration cellular when each $(Y_{i+1},Y_i)$ is a relative cell pair.
Recall that this means that $Y_{i+1}$ is obtained from $Y_i$ by attaching cells.
Here we have adopted the convention that filtrations are increasing, this is in contrast with much of the literature on the Adams spectral sequence.
This choice will clarify the $\Hi$-structure on the filtration.
We also use the term complex of spectra to mean a sequence where the composite of two successive maps has a preferred null-homotopy.
Every definition below has the necessary property that the familiar notion in complexes is recovered by applying the functor $E^0(-)$ or in spectra by applying the functor $\colim_{\bu}(-)$.

The category of filtered spectra has its own notion of homotopy as well as a natural monoidal structure.
By $c(X)_\bu$ we denote the constant filtration for any spectrum $X$.

\begin{defn}
The \underline{smash product of two filtrations} $\Xb$ and $\Yb$ is denoted by $\Gb(\Xb,\Yb)$.
The $n$th term in the filtration is 
\[
\Gamma_n(\Xb,\Yb):=\colim_{i+j\leq n}X_i\wedge Y_j.
\]
We will also denote iterated smash products of $r$-filtrations $\Xb^1,\Xb^2,\ldots,\Xb^r$ by $\Gamma^r(\Xb^1,\Xb^2,\ldots,\Xb^r)_{\bu}$.
Here, the $n$th term in the filtration is
\[
\Gamma^r_n(\Xb^1,\Xb^2,\ldots,\Xb^r):=\bigcup_{\sum_{l=1}^r\alpha_l=n} X^1_{\alpha_1} \wedge X^2_{\alpha_2}\wedge \ldots \wedge X^r_{\alpha_r}.
\]
If all of the $\Xb^i$ are the same filtration, we will use the symbol $\Gb^r$ or simply $\Gb$ when $r=2$.
\end{defn}

Here we use $\cup$ to denote a colimit.
This notation is inspired by the case where we have sequential inclusion of cell complexes.
When $\Xb$ and $\Yb$ are filtrations concentrated in positive degrees so that $X_i=Y_i=\ast$ $\forall i<0$ the above definition takes the form
\[
\Gamma_n(\Xb,\Yb)=X_0 \wedge Y_n \cup_{X_0\wedge Y_{n-1}} X_1 \wedge Y_{n-1} \cup \ldots \cup_{X_{n-1}\wedge Y_0} X_n \wedge Y_0.
\]
If both $\Xb$ and $\Yb$ are concentrated in negative degrees, such as the Adams tower, so that $X_i=X_0$ and $Y_i=Y_0$ $\forall i>0$ then we have
\[
\Gamma_n(\Xb,\Yb)=X_0 \wedge Y_n \cup_{X_0\wedge Y_{n+1}} X_{-1} \wedge Y_{n+1} \cup \ldots \cup_{X_{n+1}\wedge Y_0} X_n \wedge Y_0
\]
where $n$ is negative.
As mentioned above, this definition has the feature that
\[
E^0(\Gb(\Xb,\Yb))_n\we (E^0(\Xb)\tensor E^0(\Yb))_n:=\bigvee_{i+j=n}X_i/X_{i-1}\wedge Y_j/Y_{j-1} 
\]
where the tensor is the graded tensor product of complexes of spectra whenever the filtrations are free or Adams towers.
Given the freeness or projectivity (in the case of the Adams tower), the relevant K\"unneth spectral sequence collapses.
This fact will follow from one of the assumptions we make about the category of spectra we work with, the Geometric Condition, see \Cref{defn:geometric condition}.
Before we give the definition of multiplicative and $\Hi$-filtration, we need one more definition.

\begin{defn}
\label{defn:homotopy of filtered modules}
 Let $f_{\bu},g_{\bu}:\Ab \lra \Bb$ be two maps of filtered spectra.
 We call $H_{\bu}:\Gb(I_{\bu}, \Ab) \lra \Bb$ a \underline{filtered homotopy} from $f_{\bu}$ to $g_{\bu}$ if the following diagram commutes.
 \[
  \xymatrix{
  \Gb(c({0}_+)_{\bu}, \Ab) \we \Ab \ar[drr] \ar@/^1pc/[drrrrr]^{f_{\bu}}&
  &
  &
  &
  &
  \\
  &
  &
  \Gb(I_{\bu}, \Ab) \ar[rrr]^{H_{\bu}} \ar[rrr]&
  &
  &
  \Bb\\
  \Gb(c({1}_+)_{\bu}, \Ab) \we \Ab \ar[urr] \ar@/_1pc/[urrrrr]_{g_{\bu}}&
  &
  &
  &
  &
  }
 \]
Here, $I_{\bu}$ is the filtered spectrum coming from the standard cellular structure on the unit interval and $c(X)_{\bu}$ is the constant filtration where every map is the identity.
\end{defn}

In filtration $0$ we have $I_0:=\{0,1\}_+\we S^0\vee S^0$.
In filtration $n$ we have that $I_n:=I_+$ and the maps in the filtration are the obvious inclusions.
The associated graded of a filtered homotopy is a chain homotopy.

We can now give definitions of multiplicative and $\Hi$-filtrations.
These will give rise to multiplicative spectral sequences and extra structure, which is the whole point of this paper.
Note that the actual structure is given by some maps of filtrations, in the honest sense.
However, we only require that coherences hold up to filtered homotopy.
As we will be applying homotopy invariant functors in order to obtain our spectral sequences this will not cause any trouble.
Further, our methods of constructing multiplicative or $\Hi$-structures is only capable of giving providing coherence in the homotopy category.

\begin{defn}
We say that a filtration $\Xb$ is \underline{multiplicative} if there is a map of filtrations
\[
\xymatrix{
\Gb(\Xb,\Xb)=\Gb \ar[rr]^{\mu_{\bu}}&
&
\Xb.
}
\]
In particular, we require maps $\mu_n: \Gamma_n \to X_n$ such that
\[
\xymatrix{
\Gamma_{n-1} \ar[rr] \ar[dd]_{\mu_{n-1}}&
&
\Gamma_n \ar[dd]^{\mu_n}\\
&
&
\\
X_{n-1}\ar[rr]&
&
X_n
}
\]
commutes.
We also require that $\mu_{\bu}$ satisfy the obvious associativity condition up to filtered homotopy.
\end{defn}

Such multiplicative filtrations give spectral sequences which are multiplicative, that is that $d_r$ satisfies the Leibniz formula.
Filtrations are frequently multiplicative when the object being filtered has a (potentially weak in the sense of $A_{\infty}$) associative product structure.
When said object is a commutative or an $\Ei$-algebra, we frequently get the following notion.

\begin{defn}
\label{defn:Hinfty filtration}
We say that a filtration $\Xb$ is \underline{$\Hi$} or has an \underline{$\Hi$-structure} if there are maps of filtrations
\[
\xi_r:\Gamma_{\bu h\Sigma_r}^r:=\Gamma_{\bu}(E\Sigma_{r+}^{(\bu)},\Gb^r(\Xb))_{\Sigma_r} \lra \Xb
\]
that are compatible in the sense of Chapter 1 definition 3.1 of \cite{HRS}.
In particular, we have maps 
\[
\xi_r^n:\Gamma_{n}(E\Sigma_{r+}^{(\bu)},\Gamma_{\bu}^r)_{\Sigma_r}=\bigcup_{i+j=n}E\Sigma^{(i)}_{r+}\wedge_{\Sigma_r} (\bigcup_{\sum_{l=1}^r\alpha_l=j} X_{\alpha_1} \wedge X_{\alpha_2}\wedge \ldots \wedge X_{\alpha_r}) \to X_n.
\]
More explicitly, we require 
\[
\xi_r^{n,m}:E\Sigma^{(m)}_{r+}\wedge_{\Sigma_r}(\Gamma_n^r) \to X_{n+m}.
\]
for every $r,m,$ and $n$.
The $\xi_r^{n,m}$'s must be compatible in the sense that
\[
\xymatrix{
\displaystyle E\Sigma_{r+}^{(n)}\wedge_{\Sigma_r} \Gamma^r_{m-1} \bigcup_{E\Sigma_{r+}^{(n-1)}\wedge_{\Sigma_r} \Gamma^r_{m-1}} E\Sigma_{r+}^{(n-1)} \wedge_{\Sigma_r} \Gamma^r_m \ar[dd]_{\xi_r^{n,m-1}\cup_{\xi_r^{n-1,m-1}} \xi_r^{n-1,m}} \ar[rr]&
&
E\Sigma_{r+}^{(n)} \wedge_{\Sigma_r} \Gamma^r_m \ar[dd]^{\xi_r^{n,m}}\\
&
&
\\
X_{n+m-1} \ar[rr]&
&
X_{n+m}
\\
}
\]
commutes for all $r,n$ and $m$.
\end{defn}

This notion of an $\Hi$-filtration can be traced back to the work of Bruner in \cite{HRS}, see in particular Section 5 of Chapter 4.
We will frequently abbreviate $E\pi_{r+}^{(m)}\wedge_{\pi} \Gamma^r_n$ as $Z_{m,n}$ for $\pi\subset \Sigma_r$.
We will also use $\tGb^r$ as an abbreviation for $\Gamma^r_{\bu h\Sigma_r}$ so that
\[
\tG^r_n= \colim_{i+j\leq n} Z_{i,j}
\]
by definition, where here $\pi=\Sigma_r$. 
The compatibilities between different extended powers are only required to hold up to filtered homotopy, hence $\Hi$-structures and not $\Ei$-structures.
Clearly, every $\Hi$-filtration has an underlying multiplicative filtration by restricting to the $0$-skeleton of $E\Sigma_r$.

\begin{rmk}
One might offer a different definition where each $X_i$ is either an associative algebra or an $\Hi$-algebra.
Such a filtration would have the property that differentials commute with the operations in the strictest sense.
For example, $d_r$ would be a ring homomorphism as opposed to satisfying the Leibniz formula.
This would be a filtration \textit{by} algebras or $\Hi$-algebras which is a different variety of spectral sequence that has been useful as well, see \cite{LigaardMadsen} for example.
\end{rmk}

\begin{rmk}
\label{rmk:why naive}
At this point we must address our choice of $E\Sigma_r$.
In the motivic and equivariant settings there are multiple different choices of universal spaces with the same classical homotopy type $E\Sigma_r$.
We have chosen to work with universal spaces that have a trivial equivariant structure.
This choice is forced by the structure present on the $E_1$-page.
This spectral sequence only computes $RO(C_2)$-graded homotopy groups and the $E_2$-page is an $\Ext$-group in the classical sense for a graded commutative Hopf algebra and so our $E_1$-page only has the structure described by Bruner in \cite{HRS} and May in \cite{MaySteenrod}.
For an $\varepsilon$-commutative Hopf algebra there might be a different structure.
Another choice of universal space would give a different structure, for example the operations would change the second grading in a different fashion.
It is our belief that any such considerations will lead to other families of differentials similar to those that we present here.
\end{rmk}

\section{An outline of the method}
\label{sec:outline}
We outline the rough set up of the method that allows for such formulas.
This is just an overview to help orient the reader.
After we sketch how multiplicative filtrations have multiplicative spectral sequences we will address our situation and the necessary ingredients for \Cref{thm:diff}.

We will make use of a particular class of filtered spectra.
These filtered spectra  are denoted $U(r,s,n)_{\bu}$ and are referred to as universal examples.
The associated spectral sequence has a particularly simple form.
It has only one nontrivial differential is a $d_r$ supported on a class in filtration $s$ of geometric dimension $n$ and is an isomorphism.
This looks like
\[
\xymatrix{
\ast \ar[r]&
S^{n-1} \ar[r] \ar[dl]&
S^{n-1}\ar[rr] \ar[dl]&
&
S^{n-1}\ar[r] \ar[dl]&
D^n\ar[r] \ar[dl]&
D^n \ar[dl]\\
\ast&
S^{n-1}&
\cdots&
\ast&
S^{n} \ar@/^1pc/[lll]^{d_r'}&
\ast
}
\]
where the bottom row is the associated graded complex.
We use $d'_r$ as a place holder for the actual differential which takes suspensions into account.
The leftmost $D^n$ occurs in filtration degree $s$ and the leftmost $S^{n-1}$ occurs in filtration degree $s-r$.
These were first used by Bousfield and Kan to show that the homotopy spectral sequence for a cosimplicial simplicial set has a natural pairing in \cite{BousfieldKanUniversal}.
There are also universal examples for permanent cycles which are of the form
\[
\xymatrix{
\cdots \ar[r]&
\ast\ar[r] &
\ast \ar[r] \ar[dl]&
S^n\ar[r] \ar[dl]&
S^n \ar[dl] \ar[r]&
\cdots\\
\cdots&
\ast&
S^{n} &
\ast&
\cdots
}
\]
and are denoted $U(\infty,s,n)$.
The first $S^n$ occurs in filtration $s$.

Consider a multiplicative filtration $\Yb$ with its associated graded complex of spectra $E^0(\Yb)$.
We can use the product
\[
\mu_{\bu}: \Gb(\Yb,\Yb) \lra \Yb
\]
to recover the Leibniz formula for differentials in the associated spectral sequence.
That is to say, we want to prove that
\[
d_r(x\cdot y)= d_r(x)\cdot y + (-1)^{deg(x)}x\cdot d_r(y).
\]
Let $y_1,y_2\in E_r(\Yb)$ be two classes in the homotopy spectral sequence.
As they survive to the $r$th page of the spectral sequence, they support $d_r$'s.
These classes can be represented by maps of filtered spectra
\[
\xymatrix{
U(r,s_i,n_i) \ar[d]_{\wt{y}}&
&
\cdots \ar[r]&
\ast \ar[r] \ar[d]&
S^{n_i-1} \ar[r]^1 \ar[d]&
\cdots \ar[rr]^1 \ar[d]&
&
S^{n_i-1}\ar[r] \ar[d]&
D^{n_i}\ar[r]^1 \ar[d]&
\cdots\\
\Yb&
&
\cdots \ar[r]&
Y_{s_i-r-1} \ar[r]&
Y_{s_i-r}\ar[r]&
\cdots\ar[rr]&
&
Y_{s_i-1}\ar[r]&
Y_{s_i}\ar[r]&
\cdots
}
\]

We then smash these two maps of filtered spectra 
\[
\wt{y}_1\wedge\wt{y}_2: \Gb(U(r,s_1,n_1)_{\bu},U(r,s_2,n_2)_{\bu})_{\bu} \lra \Gb(\Yb,\Yb)_{\bu}.
\]
The associated spectral sequence of the domain is simple to describe and contains very few differentials.
This is because the factors $U(r,s_i,n_i)_{\bu}$ have particularly simple associated spectral sequences as described above. 
This coupled with the isomorphism of associated graded complexes
\[
\pi_*(E^0(\Gb(U(r,s_1,n_1),U(r,s_2,n_2))_{\bu})\iso \pi_*(E^0(U(r,s_1,n_1)_{\bu}))\bigotimes \pi_*(E^0(U(r,s_2,n_2)_{\bu}))
\]
provide a simple description of the spectral sequence associated with $\Gb(U(r,s_1,n_1)_{\bu},U(r,s_2,n_2)_{\bu})$.
The top row is the filtration $\Gb(U(r,s_1,n_1)_{\bu},U(r,s_2,n_2)_{\bu})$, the associated graded complex appears on the lower row, and $n=n_1+n_2$.
\[
\xymatrix{
\ast\ar[rr]^{\ii}&
&
S^{n-2} \ar[rr] \ar[ld]&
&
S^{n-1} \ar[rr]\ar[ld]&
&
D^{n} \ar[r] \ar[ld]&
\cdots\\
&
S^{n-2}&
&
S^{n_1-1}\wedge S^{n_2}\vee S^{n_1}\wedge S^{n_2-1}\ar@/^1pc/[ll]_{d'_r}&
&
S^{n} \ar@/^1pc/[ll]_{d'_r}&
&
}
\]
To obtain the formula for differentials we look at the differentials in this spectral sequence and push them forward along the composition
\[
\mu_{\bu}\circ (\wt{y}_1\wedge \wt{y}_2): \Gb(U(r,s_1,n_1)_{\bu},U(r,s_2,n_2)_{\bu}) \lra \Gb(\Yb,\Yb) \lra \Yb.
\]

It is this basic approach that we will use to prove \Cref{thm:diff}.
The general procedure for a spectral sequence where the input happens to be a $T$-algebra for some monad $T$ on the category of spectra is as follows .
One first shows that they can lift a $T$-algebra structure on $Y$ to a filtered $T$-algebra or a $T_{\bu}$-algebra on $\Yb$ where the filtration $\Yb$ induces the spectral sequence in question.
Next one looks at the spectral sequence for the filtered spectrum $T_{\bu}(U(r,s,n)_{\bu})$ for the universal example $U(r,s,n)_{\bu}$.
Finally, one obtains formulas for differentials on 
\[
\wt{y}:U(r,s,n)_{\bu} \sr{\wt{y}}\lra \Yb
\]
by examining the composition
\[
U(r',s',n') _{\bu} \lra T_{\bu}(U(r,s,n)_{\bu}) \sr{T_{\bu}(\wt{y})}\lra T_{\bu}(\Yb) \sr{\mu_{T_{\bu}}}\lra \Yb
\]
where $T_{\bu}(\wt{y})$ is the monad $T_{\bu}$ applied to $\wt{y}$ and $\mu_{T_{\bu}}$ is the action map of the $T_{\bu}$-algebra $\Yb$.
This composition allows one to push differentials in the spectral sequence for $T_{\bu}(U(r,s,n))$ forward to the spectral sequence for $\Yb$.

There are four main ingredients to such a result:
\begin{enumerate}
	\item Algebraic operations on the $E_2$-page of the spectral sequence that are calculable.
	This ingredient is algebraic in nature, it follows from the fact that the $E_1$-page is an $\Hi$-algebra in chain complexes.
	We recall the proof of this in \Cref{sec:power ops}.
	Our main references for this material are Chapter IV Section 3 of \cite{HRS} and \cite{MaySteenrod} in general.
	The value of this ingredient is that it allows us to identify terms in the formula with power operations that are defined algebraically.
	This gives us a greater amount of control as such operations behave more regularly than geometric or homotopical operations.
	For example, such operations satisfy the Adem relations and the Cartan formula. 
	It also gives us a place to start our obstruction theoretic construction of an $\Hi$-structure on our filtered spectrum.
	\item The filtration of the spectral sequence should have an $\Hi$-structure.
	This amounts to a collection of maps 
	\[
	\xi_r^{n,k}: E\Sigma_{r+}^{(n)}\wedge_{\Sigma_r}\Gamma_k^r(\Yb)\lra Y_{n+k}.
	\]
	that interact in the expected fashion, see \Cref{defn:Hinfty filtration}.
	We establish such an $\Hi$-structure on the $RO(C_2)$-Adams tower in \Cref{sec:Hinfty}.
	\item We then want the structure on the $E_1$-page, as a complex, to be induced by the structure on the filtration.
	This will allow us to relate the power operations on the $E_2$-page coming from the first ingredient to the structure on the abutment of the spectral sequence alluded to in the second ingredient.
	It is usually a consequence of the formalism in the underlying category.
	This ingredient is spelled out explicitly below and in \Cref{defn:geometric condition}.
	We prove that this is indeed the case in our setting in \Cref{prop: geometric condition}.
	\item We also must analyze the attaching maps of cells in the extended powers of universal examples $U(r,s,(p,q))_{\bu}$. 
	In the case of permanent cycles $U(\infty,s,(p,q))_{\bu}$ we identify such filtered spectra with stunted projective spaces with their skeletal filtration.
	We prove that
	\[
	D^{tr(n)}_2(S^{p,q})=S^{(n,0)}_+ \wedge_{\Sigma_2} (S^{p,q} \wedge S^{p,q}) \iso \sus^{p,q} \RP^{n+p,m+q} / \RP^{p-1,q}
	\]
	in \Cref{sec:Thom}.
	The attaching maps of such stunted projective spaces is computed in \Cref{sec:rppqs} as \Cref{prop:attaching maps}.
	With this, we are able to determine the coefficients of the terms in the formulas for differentials, the $\alpha$'s of Theorem \ref{thm:diff}.
\end{enumerate}

Some of these ingredients have more formal proofs than others.
In \cite{TilsonThesis}, these ingredients were implied by conditions that were assumed to hold.
Those conditions are as follows.
\begin{defn}
\label{hypotheses}
\begin{enumerate}
\item We say an equivariant spectrum $\mc{H}$ satisfies the \underline{Algebraic Condition} if $(\mc{H}_{**},\mc{H}_{**}\mc{H})$ is a commutative Hopf algebroid over a field $k$ with $\mc{H}_{**}\mc{H}$ flat over $\mc{H}_{**}$.
\item We say that an equivariant spectrum $\mc{H}$ satisfies the \underline{Homotopical Condition} if we have a natural isomorphism
\[
[X,Y\wedge \mc{H}]_{**} \iso \Hom_{\mc{H}_{**}}(\mc{H}_{**}X,\mc{H}_{**}Y).
\]
The first term is homotopy classes of maps of equivariant spectra, and the second  is maps of $\mc{H}_*$-modules.
\item We say that an extended power construction satisfies the \underline{Geometric Condition} if the following holds for any cellular filtration $Y_{\bu}\subset Y$.
Let $Z_{i,s}$ denote $E\pi^{(i)}_+\wedge_{\pi}\Gamma^r_s$ where $\pi \subset \Sigma_r$ acts on $\Gamma^r_{\bu}$ by restriction of the natural $\Sigma_r$ action.
Let $B\pi$ denote the cell complex $E\pi/\pi$.
We then require the following:
\begin{enumerate}
\item $(Z_{i,s},Z_{i-1,s})$ and $(Z_{i,s},Z_{i,s-1})$ are relative cell pairs,
\item 
\[
\displaystyle \frac{Z_{i,s}}{Z_{i-1,s}}\we\frac{B\pi^{(i)}}{B\pi^{(i-1)}}\wedge \Gamma^r_s,
\]
\item 
\[
\displaystyle\frac{Z_{i,s}}{Z_{i-1,s}\cup Z_{i,s-1}}\we \frac{B\pi^{(i)}}{B\pi^{(i-1)}}\wedge \frac{\Gamma^r_s}{\Gamma^r_{s-1}},
\]
\item and the following diagram commutes.
\[
\xymatrix{
\displaystyle\frac{Z_{i,s}}{Z_{i-1,s}\cup Z_{i,s-1}}\ar[rr]^{\we} \ar[d]_{\partial}&
&
\displaystyle\frac{B\pi^{(i)}}{B\pi^{(i-1)}}\wedge \frac{\Gamma^r_s}{\Gamma^r_{s-1}}\ar[d]^{\partial\wedge1\vee 1\wedge\partial}\\
\displaystyle\frac{Z_{i-1,s}}{Z_{i-2,s}\cup Z_{i-1,s-1}}\vee\frac{Z_{i,s-1}}{Z_{i-1,s-1}\cup Z_{i,s-2}}\ar[rr]^{\we}&
&
\displaystyle\frac{B\pi^{(i-1)}}{B\pi^{(i-2)}}\wedge \frac{\Gamma^r_s}{\Gamma^r_{s-1}} \vee \frac{B\pi^{(i)}}{B\pi^{(i-1)}}\wedge \frac{\Gamma^r_{s-1}}{\Gamma^r_{s-2}}.
}
\]
\end{enumerate}
\end{enumerate}
\end{defn}

We verify that our situation satisfies each of these conditions and show how they imply the first three ingredients above.
In \Cref{sec:power ops}, we show how the Algebraic Condition is satisfied and allows us to construct power operations and thus we have our first ingredient.
Our Homotopical Condition is an equivariant version of Bruner's Condition 3.6 in Chapter IV of \cite{HRS}.
In \Cref{sec:Hinfty}, we show that the Homotopical Condition is follows from \Cref{prop: splitting}.
We then use it as well as the Geometric Condition to show in \Cref{thm: hinfty} that the Adams tower we consider is an $\Hi$-filtration.
We also show that in our situation the Geometric Condition is satisfied in \Cref{prop: geometric condition} of in \Cref{sec:Hinfty}.
The fourth ingredient mentioned above is a subtle computation which we carry out in Sections \ref{sec:Thom} and \ref{sec:rppqs}.

Let us now outline how the above four ingredients interact to provide us with \Cref{thm:diff}.
First, consider a permanent cycle $y \in E_r^{s,n}(\Yb)$ in the spectral sequence computing $\pi_*(Y)$ from a filtration $\Yb$ of $Y$.
Represent this class as the following map of filtered spectra. 
\[
\xymatrix{
U(\infty ,s,n) \ar[d]_{\wt{y}}&
&
\cdots \ar[r] \ar[d]&
\ast \ar[r] \ar[d]&
\ast \ar[r] \ar[d]&
S^n\ar[r] \ar[d]&
S^n \ar[d] \ar[r]&
\cdots \ar[d]\\
\Yb&
&
\cdots \ar[r]&
Y_{s+2} \ar[r]&
Y_{s+1} \ar[r]&
Y_s \ar[r]&
Y_{s-1} \ar[r]&
\cdots
}
\]
Now we take the filtered extended power to obtain from $\wt{y}$ the map
\[
\tGb^r(\wt{y}):\tGb^r(U(\infty , s,n))\lra \tGb^r(\Yb).
\]
The four ingredients are now used as follows.
The $\Hi$-structure on the filtration in the second ingredient gives us a natural map
\[
\xi_r: \tGb^r(\Yb) \lra \Yb.
\]
Composing this structure map with $\tGb^r(\wt{y})$ gives us the map of filtered spectra
\[
\tGb^r(U(\infty , s,n)) \sr{\tGb^r(\wt{y})}\lra \tGb^r(\Yb) \sr{\xi_r}\lra \Yb.
\]
Thus differentials in the spectral sequence for $\tGb^r(U(\infty , s,n))$ push forward to differentials in the spectral sequence for $\Yb$.
The third ingredient allows us to relate the associated graded of $\xi_r \circ\tGb^r(\wt{y})$ with algebraic operations that exist and are well behaved by ingredient one.
The domain of this map is identified with the cellular filtration of stunted projective spaces using \Cref{prop:thom spaces}.
We then have that
\[
\hocolim_{\bu} \tGb^r(U(\infty , s,n) \iso \sus^n \RP^{\infty}_n.
\]
Now we use the computation in the fourth ingredient to compute the differentials in the spectral sequence associated to $\tGb^r(U(\infty , s,n))\we \sus^n \RP^{\bu+n}_n$.
The differentials are given by the attaching maps of $\sus^n \RP^{\infty}_n$.
These differentials push forward to the spectral sequence for $\Yb$ as desired.
This will be examined more thoroughly in \Cref{sec:equi diff}.

\section{Power operations in $\Ext$}
\label{sec:power ops}
In \cite{MaySteenrod}, May established an action of the ``big'' Steenrod algebra on the cohomology of commutative Hopf algebras, in particular it acts on the cohomology of the ``classical'' dual Steenrod algebra $\cA^{cl}_*$
\[
H^*(\cA^{cl}_*;\F_2):=Ext^{**}_{\cA^{cl}_*}(\F_2,\F_2).
\]
This ``big'' Steenrod algebra has the property that $sq^0$ does not necessarily act as the identity.
In \cite{HRS}, Bruner extended this result to give an action of the ``big'' Steenrod algebra on the cohomology of any commutative Hopf algebroid.
We can apply this result of Bruner's directly without modification.
As Bruner says, all of the necessary information is in fact contained in \cite{MaySteenrod}.

We use $k$ denote our base ring, in our situation it will be $\F_2$ (although his results certainly apply to the case when $p$ is odd).
Our Hopf algebroid is denoted $(R,A)$ and all structure maps are $k$-linear.
Note that we also require that $A$ be flat over $R$.

\begin{thm}[Bruner]
Let $(R,A)$ be a commutative Hopf algebroid over the base field $\F_2$.
There are natural homomorphisms
\[
sq^i:\Ext_{A}^{s,t}(N,M) \lra \Ext_{A}^{s+i,2t}(N,M)
\]
satisfying the Cartan formula, the Adem relations, and instability when $M$ is a commutative unital $A$-algebra in $(R,A)$-comodules and $N$ is a cocommutative unital $A$-coalgebra in $(R,A)$-comodules.
\end{thm}

The specific properties can be found in \cite{HRS}, but as we won't use the properties we won't list them here.
A familiarity with the classical Steenrod algebra will give the proper intuition.
In fact, the action of the ``classical'' Steenrod algebra can be recovered from this result if we take $N:=H_*(X)$, $M=\cA^{cl}_*$, and $(R,A)=(\F_p,\cA^{cl}_*)$.

\begin{proof}
This result follows from Lemma 2.3 in Section 2 Chapter 4 of \cite{HRS}.
The following lemma constructs the necessary structure maps so that the operations can be defined.

\begin{lemma}[Bruner]
\label{lemma: Bobs lemma}
Let $\pi$ be a subgroup of $\Sigma_r$. Let $\cV_*$ be any $k[\pi]$ free resolution of $k$ such that $\cV_0=k[\pi]$ with generator $e_0$. Let $M$ and $N$ be $A$-comodules. Let 

\[
0 \rightarrow M \leftrightarrows K_0 \leftrightarrows K_1 \leftrightarrows \cdots
\]
be an $R$-split exact sequence of $A$-comodules, 
\[
0 \rightarrow N \lra L_0 \lra L_1 \cdots
\]
a complex of extended $A$-comodules.
Let $f:M^r \lra N$ be a $\pi$-equivariant map of $A$-comodules with $\pi$ acting in the obvious fashion.
Let $\pi$ act on $K^r_*$ by permuting the factors, trivially on $L_*$, and diagonally on $\cV_* \tensor K^r_*$.
Give $\cV_* \tensor K^r_*$ the $A$-comodule structure induced by that on $K^r_*$. Then there exists a unique $\pi$-equivariant chain homotopy class of $\pi$-equivariant $A$-comodule chain maps $\Phi: \cV_* \tensor_k K^r_* \lra L_*$.
\end{lemma}

As we are working with graded $A$-comodules, the bidegree of the map $\Phi$ is important.
The bidegree of $\cV_i\tensor (K^r)_{j,t}$ is $(j-i,t)$.
This lemma, as well as the grading convention, is that of Bruner from \cite{HRS}.
It agrees with the standard depiction of operations in $\Ext$.
In practice, $L=K$ will arise as $\pi_*(E^0(\Yb))$ and so to maintain agreement the bidegree for the purpose of this article is $(i-j,t)$, but feel free to ignore this point.
We follow \cite{HRS} and use $C(A,M)_*$ to denote the reduced cobar resolution
\[
C(A,M)_s:=M\tensor_R \oo{A}^{\tensor_Rs}\tensor_R A_r 
\]
where $\oo{A}$ denotes the cokernel of the the right unit and $A_r$ denotes $A$ with $R$ acting on the left via the right unit.
This cobar resolution is used to compute $\Ext_A(N,M)$ as
\[
\Ext_A^*(N,M)=H^*(\Hom_A(N,C(A,M)_*)).
\]
We then get power operations by using \Cref{lemma: Bobs lemma} in the case where both $K$ and $L$ are $C(A,M)_*$ as follows.
We obtain
\[
\Phi_*^p:\cV\tensor_kC(A,M)^{\tensor p} \lra C(A,M)
\]
by applying the lemma to the map
\[
\widetilde{\mu}_*^p:C(A,M)_*^{\tensor p} \lra C(A,M)_*
\]
where $\wt{\mu}$ is the lift of the commutative product $\mu$ on $M$ and $\pi$ is taken to be the $p$-Sylow subgroup of $\Sigma_p$ (the cyclic group of order $p$).

We obtain operations in $\Ext$ using the composition $\theta$.
\[
\xymatrix{
\cV\tensor_k\Hom_A(N,C(A,M)_*)^{\tensor p} \ar[rr]^{\theta}\ar[dr]_{\tensor}&
&
\Hom_A(N,C(A,M)_*)\\
&
\Hom_A(N^{\tensor p},\cV\tensor_kC(A,M)_*^{\tensor p})\ar[ur]_{(\Delta^p)^*\circ(\Phi_*^p)_*} &
}
\]
where $\Delta$ is the coalgebra structure of $N$.
All of the coherences that are usually required follow from the uniqueness of the maps $\Phi^r_*$ up to homotopy. 
These coherences persist through the application of $\Hom(\Delta^r,-)$ as $\Delta$ is cocommutative.
We define the operations for $x\in \Ext_A^{s,t}(N,M)$ as
\[
sq^i(x):=\theta_*(e_{i-t+s}\tensor x\tensor x)\in \Ext_A^{s-i+t,2t}(N,M).
\]
The properties that follow can be found in \cite{MaySteenrod}.
\end{proof}

In the equivariant situation, we can apply this result directly thanks to the computations of Hu and Kriz in \cite{HuKriz}, see \Cref{subsec:recollections equivariant coh}.

\section{$\Hi$-structures on Adams filtrations}
\label{sec:Hinfty}
In this section we establish that the Adams filtration is indeed an $\Hi$-filtration.
As has been mentioned earlier, we will only consider naive $\Ei$-structures in this paper, and so $\Hi$-structure means with respect to the naive $C_2$-equivariant $\Ei$-operad.
Thus $E\pi$ will always mean a trival $C_2$-space.
Our argument follows that of Bruner and constructs the relevant maps by showing that the obstructions to existence are all in the image of the zero map.
To show that the vanishing of this obstruction is sufficient we use a result of Christensen, Dwyer, and Isaksen on obstruction theory in model categories, see \cite{ChristensenDwyerIsaksenObstModelCats}.
After stating the main results, we will recall the definition of the Adams tower as well as the canonical Adams tower.
We will also discuss the Geometric Condition which is a result of the nice formal properties of equivariant orthogonal spectra.

The main result in this section is the following theorem.
\begin{thm}
\label{thm: hinfty}
An $\eHF_2$-Adams tower $\Yb$ of an $\Hi$-ring spectrum $Y$ has an $\Hi$-structure when the ground category satisfies the Geometric Condition on extended powers and $\eHF_{2**}(Y_s)$ is projective for each $s$.
\end{thm}

In order for this result to make any sense we have to be explicit about Adams towers and the Geometric Condition for extended powers.
These definitions are adapted from \cite{HRS}.
Note that we index with negative integers so that our filtrations are increasing.

\begin{defn}
An \underline{$\eHF_2$-Adams tower} for an equivariant cellular spectrum $Y$ is a filtration
\[
\cdots \lra Y_{-2}\sr{i_{-1}}\lra Y_{-1}\sr{i_0}\lra Y_0=Y
\]
such that for each nonpositive integer $s$
\begin{itemize}
\item $Y_s/Y_{s-1}$ is a retract of $X_s \wedge \eHF_2$ for some cellular spectrum $X_s$, and
\item $\eHF_{2**}(Y_s) \to \eHF_{2**} (Y_s/Y_{s-1})$ is a $\eHF_{2**}$-split monomorphism.
\end{itemize}
\end{defn}
We always have at least one such Adams tower.
\begin{defn}
The \underline{canonical $\eHF_2$-Adams tower} is given by taking $Y_0:=Y$ and $Y_s:= Y_{s+1} \wedge \oo{\eHF}_2$.
The maps in the tower are given by smashing with $i:\oo{\eHF}_2 \to S$
\[
i_s=1\wedge i: Y_{s-1}=Y_s \wedge \oo{\eHF}_2 \to Y_s \wedge S =Y_s
\]
where $\oo{\eHF}_2$ is the fiber of the unit $S \to \eHF_2$. 
\end{defn}

That the above filtration gives an Adams tower is a classical exercise.
The first condition is satisfied as we are working in a stable model category and smash products commute with the formation of cofibers.
In fact, we can identify $Y_s/Y_{s-1}$ as $Y \wedge \oo{\eHF}^{\wedge -s}_2\wedge \eHF_2$.
The $RO(C_2)$-graded KSS collapses so that $\pi_{**}(Y\wedge \eHF_2^{s})\iso H_{**}(Y;\eF_2)\tensor_{\M_2} (\cA^{C_2}_*)^{\tensor s-1}$ for all $s\geq 1$ since $\cA^{C_2}_*$ is free over $\M_2$.
Note that $\pi_{**}\oo{\eHF_2}\wedge \eHF_2$ is free as an $\M_2$-module as well since the unit splits off, as an $\M_2$ module.
Thus the $RO(C_2)$-graded KSS also collapses for $Y_s/Y_{s-1}$ which allows us to identify the layers of the Adams tower with the terms in the classical cobar complex.
This canonical Adams tower gives rise to the Adams spectral sequence of Hu and Kriz by virtue of the flatness of the $\cA^{C_2}_*$ over the $\M_2$, see \Cref{thm:flatness} above.

The Geometric Condition that we require is as follows.
It is Lemma 5.1 in Chapter 4 of \cite{HRS}.
\begin{defn}
\label{defn:geometric condition}
We say that an extended power construction satisfies the \underline{Geometric Condition} if the following holds for any cellular filtration $Y_{\bu}\subset Y$.
Let $Z_{i,s}$ denote $E\pi^{(i)}_+\wedge_{\pi}\Gamma^r_s$ where $\pi \subset \Sigma_r$ acts on $\Gamma^r_{\bu}$ by restriction of the natural $\Sigma_r$ action.
Let $B\pi$ denote the cell complex $E\pi/\pi$.
We then require the following:
\begin{enumerate}
\item $(Z_{i,s},Z_{i-1,s})$ and $(Z_{i,s},Z_{i,s-1})$ are relative cell pairs,
\item 
\[
\displaystyle \frac{Z_{i,s}}{Z_{i-1,s}}\we\frac{B\pi^{(i)}}{B\pi^{(i-1)}}\wedge \Gamma^r_s,
\]
\item 
\[
\displaystyle\frac{Z_{i,s}}{Z_{i-1,s}\cup Z_{i,s-1}}\we \frac{B\pi^{(i)}}{B\pi^{(i-1)}}\wedge \frac{\Gamma^r_s}{\Gamma^r_{s-1}},
\]
\item and the following diagram commutes.
\[
\xymatrix{
\displaystyle\frac{Z_{i,s}}{Z_{i-1,s}\cup Z_{i,s-1}}\ar[rr]^{\we} \ar[d]_{\partial}&
&
\displaystyle\frac{B\pi^{(i)}}{B\pi^{(i-1)}}\wedge \frac{\Gamma^r_s}{\Gamma^r_{s-1}}\ar[d]^{\partial\wedge1\vee 1\wedge\partial}\\
\displaystyle\frac{Z_{i-1,s}}{Z_{i-2,s}\cup Z_{i-1,s-1}}\vee\frac{Z_{i,s-1}}{Z_{i-1,s-1}\cup Z_{i,s-2}}\ar[rr]^{\we}&
&
\displaystyle\frac{B\pi^{(i-1)}}{B\pi^{(i-2)}}\wedge \frac{\Gamma^r_s}{\Gamma^r_{s-1}} \vee \frac{B\pi^{(i)}}{B\pi^{(i-1)}}\wedge \frac{\Gamma^r_{s-1}}{\Gamma^r_{s-2}}.
}
\]
\end{enumerate}
\end{defn}

This condition, while stated with regards to filtrations, is actually concerned with how the smash product interacts with cofiber sequences and so we require no real hypotheses on the filtration other than cellularity.
The category of Lewis-May-Steinberger spectra constructed in \cite{LMS}  was the first model that had these properties.
However, techniques from model categories are able to show that modern monoidal categories of spectra satisfy the above conditions.

\begin{prop}
\label{prop: geometric condition}
The extended power construction 
\[
D_{\pi}^{tr}(X):=E\pi_+\wedge_{\pi} X^r
\]
for $\pi\subset \Sigma_r$ and $E\pi$ having the trivial $C_2$-action satisfies the above definition of the Geometric Condition.
\end{prop}

\begin{proof}
First we see that the first and second conditions follow from the fact that smash with an object takes cofiber sequences to cofiber sequences as does the orbit construction.
We begin by noting that since for each $s'$
\[
Y_{s'-1} \hra Y_{s'}
\]
is a cofibration by construction then
\[
\Gamma^r_{s-1} \hra \Gamma^r_{s}
\]
is as well.
Smashing this with
\[
E\pi^{(i-1)} \hra E\pi^{(i)}
\]
produces the commutative square
\[
\xymatrix{
E\pi^{(i-1)}\wedge \Gamma^r_{s-1}\ar@{^{(}->}[rr] \ar@{^{(}->}[dd]&
&
E\pi^{(i-1)}\wedge \Gamma^r_{s} \ar@{^{(}->}[dd]\\
&
&
\\
E\pi^{(i)}\wedge \Gamma^r_{s-1} \ar@{^{(}->}[rr]&
&
E\pi^{(i)}\wedge \Gamma^r_{s}
}
\]
of cofibrations in $C_2$-equivariant spectra.
Taking orbits with respect to the diagonal action of $\pi$ also provides us with a square 
\[
\xymatrix{
Z_{i-1,s-1}= E\pi^{(i-1)}\wedge_{\pi} \Gamma^r_{s-1}\ar@{^{(}->}[rr] \ar@{^{(}->}[dd]&
&
Z_{i-1,s}= E\pi^{(i-1)}\wedge_{\pi} \Gamma^r_{s} \ar@{^{(}->}[dd]\\
&
&
\\
z_{i,s-1}= E\pi^{(i)}\wedge_{\pi} \Gamma^r_{s-1} \ar@{^{(}->}[rr]&
&
Z_{i,s}= E\pi^{(i)}\wedge_{\pi} \Gamma^r_{s}
}
\]
of cofibrations.
From this square and the braid diagram computing the cofiber of a composition we can deduce the \underline{Geometric condition}.
\end{proof}

This can be compared with Appendix B Section 3.7 of \cite{HHRKervaire}, specifically Proposition 58.
All that is used in the above are basic facts that hold in any monoidal model category.

We also require the following result regarding maps into spectra of the form $Y\wedge \eHF_2$.
It is described as the Homotopical Condition in \Cref{hypotheses}.

\begin{prop}
\label{prop: splitting}
The $C_2$-equivariant spectrum $\eHF_2$ satisfies the Homotopical Condition in the sense that we have a natural isomorphism
\[
[X,\eHF_2\wedge Y] \iso \Hom_{\M_2}(\eHF_{2**}X,\eHF_{2**}Y)
\]
when $X$ has projective $\eHF_2$-homology and $Y$ is equivalent to a finite type $Rep(C_2)$-cell complex smashed with $\eHF_2^{\wedge s}\wedge\oo{\eHF}_2^{\wedge t}$ for $s,t\in \Z_{\geq0}$.
The first term is homotopy classes of maps of equivariant spectra, and the second  is maps of $\eHF_{2**}$-modules.
\end{prop}

\begin{proof}
We will apply the $RO(C_2)$-graded UCSS to obtain this result.
For $A$ and $B$ $\eHF_2$-modules, this spectral sequence has the form
\[
\Ext_{\M_2}^{*,*,*}(\pi_*A,\pi_*B) \ssra [A,B]_{\eHF_2}.
\]
Thus we have the spectral sequence
\[
\Ext_{\M_2}^{*,*,*}(H_{**}(X;\eF_2),H_{**}(Y;\eF_2)) \ssra [\eHF_2\wedge X,\eHF_2\wedge Y]_{\eHF_2}\iso[X,\eHF_2\wedge Y].
\]
As we have assumed that $H_{**}(X;\eF_2)$ is projective over $\M_2$ this $RO(C_2)$-graded UCSS collapses onto the zero line.
Thus we have our desired result.
\end{proof}

\begin{proof}[\Cref{thm: hinfty}]
Our proof follows that of Bruner's Theorem 5.2 in \cite{HRS}.
Note that to translate from our notation to that of Bruner in \cite{HRS} remember that our $s$ is the negative of his.
To prove this result we will construct the maps
\[
\xi_{i,s} : Z_{i,s}:=E\pi^{(i)}\wedge_{\pi} \Gamma^r_{s} \lra Y_{s+i}
\]
one at a time so that they are compatible in the sense that both diagrams
\[
\xymatrix{
Z_{i,s}\ar[rr] \ar[d]_{\xi_{i,s}}&
&
D_{\pi}^{tr} Y \ar[d]_{\xi}&
&
Z_{i-1,s} \ar[rr] \ar[d]_{\xi_{i-1,s}}&
&
Z_{i,s} \ar[rr] \ar[d]_{\xi_{i,s}}&
&
Z_{i,s+1}\ar[d]^{\xi_{i,s+1}}\\
Y_{s+i}\ar[rr]&
&
Y&
&
Y_{s+i-1}\ar[rr]&
&
Y_{s+i}\ar[rr]&
&
Y_{s+i+1}
\\
}
\]
commute.
This map will be constructed by showing that the obstruction to its existence is in the image of the zero map.

Note that we already have the maps
\[
\xi_{0,s}:Z_{0,s}=\Gamma^r_s\lra Y_s.
\]
These exist because $\Yb$ is a multiplicative filtration.
The proof that this filtration is multiplicative is identical to that of Bruner, see Chapter 4 Section 4 of \cite{HRS}.
All that is necessary is the flatness and projectivity, which we indeed have in our situation.
Thus these $\xi_{0,s}$ provide for the beginning of our inductive construction.
Suppose we already have constructed $\xi_{i,s}$ suitably for all $i<k$.
Whenever $-s<k$ we define $\xi_{k,s}$ to be the composite
\[
Z_{k,s} \hra E\pi_+\wedge_{\pi} \Gamma^r_0= D_{\pi}^{tr} Y \lra Y=Y_0
\]
which exists as $Y$ is an $\Hi$-ring spectrum.
Therefore we can assume that $\xi_{k,s'}$ has been constructed for all $s<s'$.
We then must construct $\xi_{k,s}$ compatibly, which amounts to ensuring that the diagram
\[
\xymatrix{
Z_{k-1,s} \ar[dd] \ar[rr]&
&
Z_{k-1,s+1} \ar[dd] \ar[rrr]_{\xi_{k-1,s+1}}&
&
&
Y_{s+k} \ar[dd]\\
&
&
&
&
&
\\
Z_{k,s} \ar[rr] \ar@{-->}[uurrrrr]^{\xi_{k,s}}&
&
Z_{k,s+1} \ar[rrr]^{\xi_{k,s+1}}&
&
&
Y_{k+s+1}
}
\]
commutes.

The obstruction to the existence of such a $\xi_{k,s}$ making the above diagram commute exists.
In \cite{ChristensenDwyerIsaksenObstModelCats}, Christensen, Dwyer, and Isaksen establish a general theory of obstruction theory in model categories..
In Section 8 of that paper, they show that our setting has such an obstruction theory.
The general machinery of that paper can then be applied to see that the obstruction to the existence of such a map lies in the group of homotopy classes of maps
\[
\mathcal{O}(\xi_{k,s})\in [Z_{k,s}^1,Y_{k+s}^1]
\]
where
\[
Z_{k,s}^n:=cof(Z_{k-n,s}\ra Z_{k,s})
\]
and
\[
Y_{k+s}^n:=cof(Y_{k+s}\ra Y_{k+s+n}).
\]
This obstruction is alsonatural.

Recall that for the canonical $\eHF_2$-Adams tower we have that
\[
Y^1_{k+s}\we Y \wedge \oo{\eHF}^{\wedge -(k+s)}_2\wedge \eHF_2.
\]
If $\eHF_{2**}(Z^1_{k,s})$ is projective then we can use the above result to identify the group containing the obstruction
\[
\mathcal{O}(\xi_{k,s})\in [Z^1_{k,s},Y^1_{k+s}]
\]
with
\[
\Hom_{\M_2}(\eHF_{2**}(Z^1_{k,s}),\pi_{**}(Y^1_{k+s})).
\]
Using part 2 of the Geometric condition we will see that $\eHF_{2**}(Z^1_{k,s})$ is a projective $\M_2$-module.
This follows from a couple of different results.
The Freeness Theorem of Kronholm, \Cref{thm:flatness} of Hu and Kriz paired with the $RO(C_2)$-graded KSS imply that
\[
\eHF_{2**}(Z^1_{k,s}) = \eHF_{2**}(\frac{B\pi^{(i)}}{B\pi^{(i-1)}}) \tensor_{\M_2} \eHF_{2**}(\Gamma^r_s).
\]

$\eHF_{2**}(\Gamma^r_s)$ is projective as it is a summand of a free module.
To see this note that the map 
\[
\Gamma^r_s \lra \Gamma^r_{s+1}
\]
is $0$ in $\eHF_2$-homology as it is a pushout of smash products of maps that induce $0$ in $\eHF_2$-homology.
This uses the collapsing of the KSS for the $\eHF_2$-homology of each $Y_s$.
Thus we have a short exact sequence
\[
0\lra \eHF_{2**}(\Gamma^r_{s+1})\lra \eHF_{2**}(\Gamma^r_{s+1}/\Gamma^r_s)\lra \eHF_{2**}(\sus^{1,0}\Gamma^r_s) \lra 0
\]
of $\M2$-modules.
Recall that $\eHF_{2**}(\Gamma^r_{s+1}/\Gamma^r_s)$ is a term in $\eHF_{2**}(E^0(\Gb^r)$ which is just $\eHF_{2**}(E^0(\Yb))^{\tensor r}$ by another collapsing KSS.
The above sequence is split and $\eHF_{2**}(\Gamma^r_{s+1})$ is projective over $\M_2$ whenever $\eHF_{2**}(\sus^{1,0}\Gamma^r_s)$ is projective.
As $\Gamma_0^r=Y_0^r$ has projective $\eHF_2$-homology by assumption and the KSS.
Thus each $\Gamma^r_s$ has projective $\eHF_2$-homology as an $\M_2$-module by induction.

Now we have the following string of isomorphisms
\begin{align*}
[Z^1_{k,s},Y^1_{k+s}] & \iso [Z^1_{k,s},\bigvee_{\alpha_i}\sus^{\alpha_i}\eHF_2]\\
 & \iso \bigoplus_{\alpha_i}\eHF_2^*(\sus^{-\alpha_i}Z^1_{k,s}) \\
 & \iso \Hom_{\M_2}(\eHF_{2**}(Z^1_{k,s}),\pi_{*}(Y^1_{k+s})).
\end{align*}
The naturality of the obstruction implies that it is in the image of the map induced by 
\[
1 \wedge j_s: Z^1_{k,s} \lra Z^1_{k,s+1}
\]
by part 2 of the Geometric Condition, where
\[
j_s: \Gamma^r_s \lra \Gamma^r_{s+1}
\]
is the map in the Adams tower $\Gb^r$.
Therefore $1 \wedge j_s$ induces $0$ in homology.
Thus the obstruction vanishes as it is in the image of the $0$-map.
\end{proof}

\section{Extended powers of spheres and Thom complexes}
\label{sec:Thom}
In this section we identify the extended powers of representation spheres with truncated equivariant projective spaces.
In the next section we analyze the cell structure of these extended powers.
This will determine the differentials in spectral sequence for extended powers of permanent cycles.
This is the meat of ingredient four.
The identification of extended powers of spheres with truncated projective spaces goes back to Atiyah.
Our proof is simply a realization that the classical proof works equivariantly.
As there will be two groups of order two floating around we will distinguish between the external and internal actions by writing $\Sigma_2$ for the external action and $C_2$ for the internal action.
Note also that this result is where there are some difficulties in moving to the motivic setting.

\begin{defn}
We denote by $D^{tr}_2(X):=S(\R^{\infty,0})_+\wedge_{\Sigma_2}(X\wedge X)$ and $D^{\infty}_2(X):=S(U)_+\wedge_{\Sigma_2}(X\wedge X)$ where $U$ is the complete $C_2$-universe $\dsum_i \R^{2,1}$.
We will use 
\[
D^{tr(2n)}_2(X):=S(\R^{2n+1,0})_+\wedge_{\Sigma_2}(X\wedge X)
\] and 
\[
D^{\infty(2n)}_2(X):=S(\R^{2n+1,n})_+\wedge_{\Sigma_2}(X\wedge X)
\]
to denote the topological $2n$-skeletons of the respective extended power constructions. 
\end{defn}
We will in fact only use the even skeletons as this is the most natural way of filtering the complete universe $U$.
If more is desired then a choice must be made, and there are preferred choices.
These are dictated by wanting the cohomology to be easy to compute and so we alternate between adding a trivial representation cell and a non trivial representation cell.
This choice forces the cellular spectral sequence, see for example \cite{KronholmFree}, to collapse immediately.

The following can be found in \cite{HRS}, but we recall its proof as it will extend to the equivariant setting.

\begin{prop}
$D_2 S^n \iso \sus^n \RP^{\infty}/\RP^{n-1}$ and further $D^{(k)}_2 S^n \iso \sus^n \RP^{n+k}/\RP^{n-1}$.
\end{prop}

\begin{proof}
By definition, $D_2 S^n = E\Sigma_{2+} \wedge_{\Sigma_2} S^n \wedge S^n$ where $\Sigma_2$ acts on $S^n \wedge S^n$ via the twist map.
We will prove the result by identifying this quadratic construction as the Thom complex of a bundle over $B\Sigma_2=\RP^{\infty}$.
If $\xi$ is an $n$-plane bundle over $BG$ then we can realize the total space of $\xi$ as $EG \times_G V$ where $V$ is a $n$-dimensional $G$ representation (with a $G$ invariant inner product).
This extends to the associated disk and sphere bundles of $\xi$ giving isomorphisms 

\[D(\xi) \iso EG \times_G D(V) \  \  and \ \ S(\xi) \iso EG \times_G S(V).
\] 
Recall that the Thom construction is given by 
\[
Th(\xi)=\frac{D(\xi)}{S(\xi)}.
\]
We have the following chain of homeomorphisms.
\begin{center}
$\displaystyle\frac{EG \times D(V)}{EG \times S(V)}\iso \frac {EG \times D(V)/S(V)}{EG \times \ast} \iso \frac{EG \times S^V}{EG \times \ast} \iso \frac{EG_+ \times S^V}{EG_+ \vee S^V} \iso EG_+ \wedge S^V$
\end{center}
It is equivariant with respect to the diagonal action.
Thus, we have 
\[
Th(\xi) \iso EG_+ \wedge_G S^V.
\]

So now we wish to identify $D_2 S^n$ as the Thom space of a bundle.
Consider the $2n$-plane bundle $\gamma$ over $B\Sigma_2$ with fiber $\R^n \oplus \R^n$ having $\Sigma_2$ act by the twist map.
This is equivalent the tensor product $\R^n \tensor \R^2$ where the first factor is the $n$-dimensional trivial representation and the second is the regular representation.
The regular representation of $\Sigma_2$ decomposes and so we also have that $\gamma$ decomposes as
\[
\gamma \iso n\epsilon \dsum n\lambda
\]
where $\epsilon$ is the trivial line bundle and $\lambda$ is the canonical line bundle.
Note that all of this works over a finite skeleton of $B\Sigma_2$; that is, as bundles over $\RP^k$.
As such, we will now work implicitly over $\RP^k$ and all bundles will be restricted without a change in notation.

The Thom complex $Th(n\epsilon \oplus n\lambda)$ is just $\sus^n Th(n\lambda)$.
We compute $Th(n\lambda)$ by realizing that $n\lambda$ is the normal bundle $\nu$ to the standard embedding of  $\RP^k$ in $\RP^{n+k}$.
This embedding of $\RP^k$ in $\RP^{n+k}$ is gotten by sending every $k+1$-tuple 
\[
[x_0:x_1: ...:x_k] \mapsto [x_0:x_1:...:x_k:0:...:0].
\]
Any point in the total space of $\nu$ is of the form $[x_0:x_1:...:x_k:a_1:a_2:...:a_n]$. 
The isomorphism of $\nu$ with $n\lambda$ is given by taking $a_i$ to be the $i$th coordinate of $n\lambda$.
The total space of $\nu$ is all of $\RP^{n+k}$ except for the orthogonal $\RP^{n-1}$, points with homogeneous coordinates $[0:...:0:x_{k+1}:...:x_{n+1}]$.
Therefore we have
\[
Th(\nu) \iso \frac{\RP^{n+k}}{\RP^{n-1}}
\]
by the Pontryagin-Thom construction which identifies the Thom space of the normal bundle to an embedding $i: X \hra Y$ with the quotient of the ambient space by the complement of the normal bundle.
We now have our desired homeomorphism
\[
S^k_+\wedge_{\Sigma_2}(S^n\wedge S^n)\iso \sus^n \frac{\RP^{n+k}}{\RP^{n-1}}
\]
as the composition
\[
S^k_+\wedge_{\Sigma_2}(S^n\wedge S^n)\iso Th(n\epsilon \dsum n\lambda) \iso \sus^n Th(\nu) \iso \sus^n \frac{\RP^{n+k}}{\RP^{n-1}}.
\]
\end{proof}

We are interested in the above result for the equivariant spheres $S^{p,q}$.
The above proof works equally well in the equivariant category.
Replace all of the $\Sigma_2$-representations with $\Sigma_2$-representations in the category of $C_2$-spaces.
Here, we use $\widehat{E}$ to denote the ``universal space'' construction in the category of $C_2$-spaces, and the superscript denotes which skeleton we are considering.
There are (at least) two different possible models for $\widehat{E}\Sigma_2$ in $C_2$-spaces, one with a trivial $C_2$-action and one coming from the complete universe.
The ``trivial'' case is obtained by considering $S^{\infty,0}$ whereas the other extreme is gotten by considering $S^{2\infty,\infty}:=\cup_n S^{2n,n}$.
We will use the symbols $\widehat{E}\Sigma^{(n,m)}_2$ and $S^{n,m}$ interchangeably in the following proposition.
The following equivariant projective spaces $\RP^{p,q}$ are the spaces of $1$-dimensional subspaces of $\R^{p+1,q}$.
We will say more about them in the next section.

\begin{prop}
\label{prop:thom spaces}
$\widehat{E}\Sigma^{(n,m)}_{2+} \wedge_{\Sigma_2} (S^{p,q} \wedge S^{p,q}) \iso \sus^{p,q} \RP^{n+p,m+q} / \RP^{p-1,q}$
\end{prop}
\begin{proof}
We want to identify $S^{n,m}_+ \wedge_{\Sigma_2} (S^{p,q} \wedge S^{p,q})$ as some truncated equivariant projective space, where $\Sigma_2$ is acting on $S^{n,m}$ via the antipodal action on the underlying $S^n$.
As above, we have the identification
\[
S^{n,m}_+ \wedge_{\Sigma_2} (S^{p,q} \wedge S^{p,q}) \iso Th(\R^{p,q}\tensor \epsilon \dsum \R^{p,q} \tensor \lambda_{n,m})
\]
where $\lambda_{n,m}$ is the canonical line bundle on $\RP^{n,m}$.
This bundle has an internal $C_2$ action: as $\RP^{n,m}$ is obtained from $\R^{n+1,m}$, it inherits a $C_2$ action, which permutes lines in the underlying $\RP^n$.
Just as above, we have that
\[
Th(\R^{p,q}\tensor \epsilon \dsum \R^{p,q} \tensor \lambda_{n,m}) \iso \sus^{p,q} Th(\R^{p,q} \tensor \lambda_{n,m}) \iso \sus^{p,q} Th( \nu(\RP^{n,m} \hookrightarrow \RP^{n+p,m+q})).
\]
In order to have the right $C_2$-equivariant structure, we embed $\RP^{n,m}$ as the subspace of points with homogenous coordinates of the form 
\[
[x_0:x_1:\ldots:x_{n-m}:0:\ldots:0:x_{n+p+1-m-q}:x_{n+p+2-m-q}:\ldots:x_{n+p-q}:0:\ldots:0] \in \RP^{n+p,m+q}.
\]
The complement to the normal bundle of this embedding is then $RP^{p-1,q}$ embedded as 
\[
[0:0:\ldots:0:x_{n-m+1}:\ldots:x_{n+p-m-q}:0:0:\ldots:0:x_{n+p+1-q}:\ldots:x_{n+p}] \in \RP^{n+p,m+q}.
\]
Therefore the Thom space of this normal bundle is 
\[
Th(\nu(\RP^{n,m} \hookrightarrow \RP^{n+p,m+q}))\iso \frac{\RP^{n+p,m+q}}{\RP^{p-1,q}}.
\]
\end{proof}

We now have two models of the quadratic construction.
For our purposes, we will be more interested in the ``naive'' quadratic construction which we denote by $D^{tr}_2(-)$, see \Cref{rmk:why naive}.
The other one may provide for formulas in other equivariant Adams spectral sequences which we won't speculate about here.

\begin{cor}
\label{cor:equiv extended powers}
$D^{tr}_2(S^{p,q})\iso \sus^{p,q} \RP^{\infty+p, q}/\RP^{p-1,q}$ and $D^{\infty}_2(S^{p,q}) \iso \sus^{p,q} \RP^{\infty+p, \infty+q}/\RP^{p-1,q}$.
\end{cor}
This formulation is not as clear as one might hope.
Instead, let us focus on the $2$-skeletons.
\begin{center}
$D^{tr(2)}_2(S^{p,q})\iso \sus^{p,q}\frac{\RP^{p+2, q}}{\RP^{p-1,q}}$ and $D^{\infty (2)}_2(S^{p,q}) \iso \sus^{p,q}\frac{\RP^{p+2,q+1}}{\RP^{p-1,q}}$.
\end{center}
While we only need to consider the topological $1$-skeleton in order to investigate the differentials, the difference between the two complexes may be easier to see when we consider the $2$-skeleton given the cell structure of $\widehat{E}\Sigma_2^{(n,m)}$.
It will always be best to consider the even dimensional topological skeletons.
In that situation, we have a good model for $\widehat{E}\Sigma_2^{(2n)}$, namely $S(\displaystyle\bigoplus_{i=1}^n\R^{2,1})$.

\section{Equivariant projective spaces}
\label{sec:rppqs}
In this section we will compute enough about the cell structure of these projective spaces in order to produce the desired formulas for differentials.
This amounts to understanding the attaching maps
\[
S^n \to X_n \to X_n/X_{n-1}\we \bigvee S^n
\]
as an element of $\bigoplus \pi_0(S^0)$.
This will be the coefficient of the differentials in \Cref{sec:equi diff}.

\begin{defn}
Let $\RP^{p,q}$ denote the space of 1-dimensional subspaces of $\R^{p+1,q}$.
\end{defn}
Another model for $\RP^{p,q}$ is as the quotient of $\R^{p+1,q}$ by the action of $\R^{\times}$.
One can also restrict to the unit sphere $S^{p-1,q}$ and quotient out by the antipodal action.
The $C_2$ action is inherited from the $C_2$ action on $\R^{p,q}$.
If we consider a point $[x_0:x_1:\cdots:x_{p}]\in \RP^{p,q}$, then $C_2$ acts by negating the last $q$ coordinates.
The following observations will be useful later in our analysis.

\begin{lemma}
The equivariant real projective spaces $\RP^{p,q}$ and $\RP^{p,p-q+1}$ are identical as $C_2$-spaces.
\end{lemma}
\begin{lemma}
$\uu(\RP^{p,q})=\RP^p$ and $\ff(\RP^{p,q})=\RP^{q-1} \coprod \RP^{p-q}$.
\end{lemma}

Classically, $\pi_0(S^0)$ is just the integers.
Equivariantly, there is another generator $\varepsilon$ obtained by desuspending the twist map 
\[
S^{1,1}\wedge S^{1,1} \to S^{1,1}\wedge S^{1,1}.
\]
This brings us to the following computation from \cite{Bredon}.
\begin{thm}[Bredon]
\label{thm: pi_00}
We have that $\pi_{0,0}(S^{0,0}) \iso \Z^2$ generated by $1$ and $\varepsilon$.
\begin{eqnarray*}
deg(\uu(1)) & = & 1\\
deg(\ff(1)) & = & 1\\
deg(\uu(\varepsilon)) & = & -1\\
deg(\ff(\varepsilon)) & = & 1.
\end{eqnarray*}
\end{thm}

\begin{rmk}
We will use the above theorem to determine the attaching maps of $Rep(C_2)$-cells in projective spaces.
This is done by looking at the effect of taking fixed points and forgetting the $C_2$-structure on the attaching map.
Such functors commute with the formation of cofiber sequences in the unstable category of $C_2$-equivariant cell complexes.
In the stable category though, there are multiple notions of fixed points.
We employ the notion of geometric fixed points, see Section 7 of \cite{SchwedeEquivariant}.
This construction, as mentioned in \Cref{subsec:CWcplexes}, has the property that taking geometric fixed points of a suspension spectrum is equivalent to taking the suspension spectrum of the fixed points of a based $G$-CW-complex.
It is also the case that one can utilize geometric fixed points to detect equivariant weak equivalences.
So in the proof of \Cref{prop:attaching maps} we proceed by computing attaching maps in the unstable category and then determining which element of $\pi_{0,0}S^{0,0}$ we get by stabilizing.
\end{rmk}

\begin{prop}
\label{prop:attaching maps}
As $C_2$-equivariant cell complexes, 
\[
\RP^{p,q}\iso \RP^{p-1,q}\cup_{f^1_{p-1,q}}CS^{p-1,q}
\]
and
\[
\RP^{p,q}\iso \RP^{p-1,q-1}\cup_{f^2_{p-1,p-q+1}}CS^{p-1,p-q+1}.
\]
Both such attaching maps give the same element of $\pi_{0,0}S^{0,0}$, whose common value we denote by $\widetilde{f}_{i,j}$. 
The attaching map of the $(i,j)$-cell, is given by
\[
\widetilde{f}_{i,j} = \left\{
\begin{array}{ll}
1-\varepsilon & i \equiv 1, j \equiv 1 \pmod{2} \\
2  & i \equiv 1, j \equiv 0 \pmod{2} \\
1+\varepsilon & i \equiv 0, j \equiv 1 \pmod{2} \\
0 & i \equiv 0, j \equiv 0 \pmod{2} \\
\end{array}
\right\}
\]
\end{prop}

\begin{proof}
This result is proved by using $\uu$ and $\ff$ to first determine the cofibers of the inclusions and then computing the attaching maps using \Cref{thm: pi_00} and the fact that taking geometric fixed points commutes with taking suspension spectra for based $G$-cell complexes.
Also, we can assume that $q\leq \frac{p+1}{2}$.
If not, we work with the equivalent space $\RP^{p,p-q+1}$ which has fewer nontrivial $C_2$-cells.

Now we take the first case, attaching a cell to $\RP^{p-1,q}$ along a map we call $f^1$ to obtain $\RP^{p,q}$.
We fix $\RP^{p-1,q}$ inside $\RP^{p,q}$ as the last $p+q$ homogeneous coordinates to insure that the equivariant structure is correct.
As mentioned above, we use the fact that both $\uu$ and $\ff$ preserve cofiber sequences to help us compute the cofiber  
\[
C:=Cof(\RP^{p-1,q} \hra \RP^{p,q}).
\]
We see easily that $\uu(C)=S^p$.
The inclusion 
\[
\RP^{p-1,q} \hra \RP^{p,q}
\]
 induces
\[
\RP^{q-1} \coprod \RP^{p-q-1} \hra \RP^{q-1} \coprod \RP^{p-q}
\]
upon taking fixed points and so $\ff(C)=S^{p-q}$.
This implies that $C=S^{p,q}$, therefore $\RP^{p,q}$ is the result of attaching a cell along a map 
\[
f^1_{p-1,q}:S^{p-1,q} \to \RP^{p-1,q}.
\]

Now we turn to determining the projection of this map onto the top cell.
This is given by the composition
\[
\widetilde{f}^1_{p-1,q}: S^{p-1,q} \hra \RP^{p-1,q} \lra  S^{p-1,q}.
\]
We determine which element of $\pi_{0,0}(S^{0,0})$ we have by computing the degree of the map on the underlying space and on fixed points.
These two values uniquely determine which map we have up to equivariant homotopy. 
The map 
\[
\uu(\widetilde{f}^1_{p-1,q}): S^{p-1} \to \RP^{p-1} \to S^{p-1}
\]
has $deg(\uu(\widetilde{f}^1_{p-1,q}))=1+(-1)^p=1-(-1)^{p-1}$.
Now, 
\[
\ff(\widetilde{f}^1_{p-1,q}): S^{p-q-1} \to \RP^{q-1} \coprod \RP^{p-q-1} \to S^{p-q-1}
\]
lands in the second term of the coproduct.
This implies that $deg(\ff(\widetilde{f}^1_{p-1,q}))=1+(-1)^{p-q}=1-(-1)^{p-1+q}$.
Therefore we have the formula
\[
deg(\uu(\widetilde{f}^1_{i,j}),\ff(\widetilde{f}^1_{i,j}))=(1-(-1)^i,1-(-1)^{i+j}).
\]

We are also interested in obtaining $\RP^{p,q}$ from $\RP^{p-1,q-1}$.
Consider the underlying space and the fixed points of 
\[
C:=Cof(\RP^{p-1,q-1} \hra \RP^{p,q}).
\]
We have that $\uu(C)=S^p$.
The map on fixed points induced by the inclusion is given by 
\[
\RP^{q-2} \coprod \RP^{p-q} \hra \RP^{q-1} \coprod \RP^{p-q}
\]
and so $\ff(C)=S^{q-1}$.
Therefore, $C=S^{p,p-q+1}$ and so we can realize $\RP^{p,q}$ as the cofiber of a map 
\[
f^2_{p-1,p-q+1}:S^{p-1,p-q+1} \to \RP^{p-1,p-q+1} \iso \RP^{p-1,q-1}.
\]

Consider the composition 
\[
\widetilde{f}^2_{p-1,p-q+1}: S^{p-1,p-q+1} \hra \RP^{p-1,p-q+1} \lra S^{p-1,p-q+1}.
\]
We see that 
\[
\uu(\wt{f}^2_{p-1,p-q+1}): S^{p-1} \to \RP^{p-1} \to S^{p-1}
\]
has $deg(\uu(\wt{f}^2_{p-1,p-q+1}))=1+(-1)^p=1-(-1)^{p-1}$.
Also, 
\[
\ff(\wt{f}^2_{p-1,p-q+1}): S^{q-2} \to \RP^{q-2} \coprod \RP^{p-q} \to S^{q-2}
\]
factors through $\RP^{q-2}$ for the same reasons as above.
Therefore, $deg(\ff(\wt{f}^2_{p-1,p-q+1}))=1+(-1)^{q-1}=1-(-1)^{(p-1)+(p-q+1)}$.
Thus we arrive at the formula
\[
deg(\uu(\wt{f}^2_{i,j}),\ff(\wt{f}^2_{i,j}))=(1-(-1)^i,1-(-1)^{i+j}).
\]

We obtain $\wt{f}$ by computing the degree of $\uu(f)$ and $\ff(f)$ to express it as a linear combination of $1$ and $\varepsilon$.
\end{proof}

It is possible to obtain $\RP^{p,q}$ in different ways by attaching different equivariant cells.
The above result will give us what is necessary to construct the real projective spaces via a preferred sequence of cell attachments.
Any model for the cellular tower of $\RP^{p,q}$ will give rise to a spectral sequence which computes the $RO(C_2)$-graded (co)homology of $(\RP^{p,q})$, see \cite{KronholmFree}.

\begin{cor}
Both the $RO(C_2)$-graded homology and cohomology of $\RP^{p,q}$ with coefficients in $\eF_2$ are free $\M_2$-modules.
\end{cor}
Our particular preferred model, alluded to in the above proposition, will insure that all of the differentials in the cellular spectral sequence are trivial.

\begin{proof}
Without loss of generality we can assume that $q<\frac{p+1}{2}$ and $q\neq 0$.
Now we alternate between attaching trivial and nontrivial $C_2$-representation cells until we have attached $q$-twisted cells.
The projections of the attaching maps, which induce the differentials in the cellular spectral sequence of Kronholm, are as described in \Cref{prop:attaching maps}.
Once we have attached $q$ twisted cells we attach continue attaching trivial cells.
All of the differentials in this spectral sequence are $0$.
The differentials are all $\M_2$-module maps.
They either have to be zero because of degree reasons, when we attach a twisted cell there is nothing to hit, or they are maps that have trivial Hurewicz image since $\varepsilon\pi_{0,0}S^{0,0}$ is sent to $1\in\M_2$.
\end{proof}

The real point of the filtration in the above proof is making sure that we attach cells in such a way that there can never be a differential hitting something in the negative cone of a suspension of $\M_2$, that would make things more complicated.

\section{Motivic results}
\label{sec:motivic}
In this section we will outline how the above results can be extended to the motivic setting.
We will establish that all of the ingredients discussed in \Cref{sec:outline} hold in the motivic setting.
In fact, some of the above equivariant results are necessary for our motivic results.
In particular, the results of \Cref{sec:rppqs} are used directly to describe the attaching maps of motivic extended powers of spheres.
We realize that in light of recent work of Bachmann and Hoyois, see \cite{BachmannHoyoisNorms}, this statement is potentially ambiguous and so we will clarify that we only mean naive extended powers.
Specifically, the model of the total spaces $E\Sigma_n$ that we use are all simplicial.

Before extending the above results, we will recall a few basic facts about motivic homotopy theory over $Spec(\R)$.
As in the equivariant setting, that motivic cohomology satisfies the algebraic condition follows from a computation, due to Voevodsky in \cite{VoevodskyMotivic11} and \cite{VoevodskyEM12}.
Similarly, the geometric condition follows from model categorical properties of our chosen model of motivic homotopy theory, see \cite{PaninPimenovRoendigsmodel} for specifics about the model structure we consider.
After this we will explain how the results of \Cref{sec:Thom} can be sidestepped using the equivariant case and Betti realization.
Finally, the homotopical condition and $\Hi$-structures in the motivic setting are discussed.

We have chosen to present the material in this fashion in order that the proof in the next section, for both cases, has all of the necessary results established prior to it.

\subsection{Recollections from motivic homotopy theory over $\R$.}
Here we recall some basic facts about motivic homotopy theory over $Spec(\R)$.
There are many good introductions to this topic, such as \cite{MotivicIntroBook} or \cite{DuggerIsaksenMotivicCell}, and so we will be brief.
Recent work of Heller and Ormsby, see \cite{HellerOrmsby}, relating $C_2$-equivariant homotopy theory and motivic homotopy theory over $Spec(\R)$ will be useful to us.
However, their main theorem can not be used to produce our desired results since the motivic bigraded spheres are not in the image of their section of Betti realization.
When necessary, we will distinguish the motivic spheres from the equivariant ones by use of the subscripts $_{\R}$ and $_{C_2}$ respectively.

Given a scheme $X$ defined over $k$ and an embedding of $k$ into $\C$ one can obtain a $Gal(\C/k)$-equivariant topological space by considering the complex points of $X$ base changed to $\C$.
This extends to the realization functors of motivic homotopy theory which we will make crucial use of.
The Betti realization functor
\[
Be:SH(\R) \lra SH(C_2)
\]
from the motivic stable homotopy category over $Spec(\R)$ to the $C_2$-equivariant stable homotopy category takes $S^{p,q}_{\R}$ to $S^{p,q}_{C_2}$ and $\HFR_2$ to $\eHF_2$.
Further, $Re$ is monoidal and takes motivic cofiber sequences to equivariant cofiber sequences.
The most complete reference for this is the work of Heller and Ormsby, see Section 4 of \cite{HellerOrmsby}.
Heller and Ormsby work with the closed flasque model structure, as developed in \cite{PaninPimenovRoendigsmodel}.
Their main reasons for choosing this model structure is that all of the standard motivic spheres are cofibrant, various base change functors as well as the Betti realizations are left Quillen functors.
They establish the following Quillen adjunction.

\begin{prop}[Heller-Ormsby, Proposition 4.8 of \cite{HellerOrmsby}]
The functors
\[
Re^{C_2}_{B}: Spt^{\Sigma}_{\mathbb{P}^1}(\R) \leftrightarrows Spt^{\Sigma}_{S^{2,1}}(C_2): Sing_B^{C_2}
\]
form a Quillen adjoint pair and $Re^{C_2}_{B}$ is strong symmetric monoidal.
\end{prop}

We apologize for the mix of their notation with ours.
The fact that $Re^{C_2}_{B}$ is strong symmetric monoidal implies that $Re^{C_2}_B(S^{p,q}_{\R})=S^{p,q}_{C_2}$.
Heller and Ormsby also prove the following important result.

\begin{thm}[Heller-Ormsby, Theorem 4.17 of \cite{HellerOrmsby}]
There is an isomorphism
\[
\mathbb{L}Re^{C_2}_B(\HH A^{\R})\iso \HH \underline{A}
\]
in the $C_2$-equivariant stable homotopy category for any abelian group $A$.
\end{thm}

Since both spectra are cellular, this implies in particular that $Re^{C_2}_B(\HFR_2)\we\eHF_2$.
Note that the cellularity of $\HFR_2$ is a nontrivial result which follows from the work of Hoyois in \cite{HoyoisHopkinsMorel}.
Therefore, the Betti realization of $\HFR_2$ is cellular as well and the isomorphism in the homotopy category lifts to an honest equivalence.
We will only be working with cellular motivic spectra.
Therefore, by Proposition 7.1 of \cite{DuggerIsaksenMotivicCell} weak equivalences are determined by bigraded homotopy groups.
This was key in our equivariant work.

Thus we will use Betti realization to compute the attaching maps of motivic extended powers of spheres from the equivariant attaching maps and the fact that $Re$ induces an isomorphism $\pi_{0,0}S^{0,0}_{\R}\iso \pi_{0,0}S^{0,0}_{C_2}$.
See \Cref{subsec: motivic extended powers of spheres} for details.

\subsection{The Algebraic condition in the motivic setting}
\label{subsec: motivic algebraic condition}
The following facts, as in \Cref{sec:power ops}, are sufficient for establishing that motivic cohomology with coefficients in $\F_2$ satisfy the algebraic condition.

\begin{thm}[Voevodsky \cite{VoevodskyMotivicEM}]
\label{thm: motivic steenrod}
The dual motivic Steenrod algebra over $Spec{\R}$ is given by
\[
\cA^{\R}_*:=\pi_{*,*}(\HFR_2\wedge \HFR_2)\iso \pi_{*,*}(\HFR_2)[\tau_0,\tau_1,\ldots, \xi_1,\xi_2,\ldots]/(\tau_k^2=\tau \xi_{k+1}+\rho\tau_{k+1}+\rho\tau_0\xi_{k+1})
\]
where we recall that $\pi_{*,*}(\HFR_2)\iso \F_2[\tau,\rho]$, which we sometimes denote by $\M_2^{\R}$.
\end{thm}
The Hopf algebra structure maps can be found in Dugger-Isaksen \cite{DuggerIsaksenLowMWstems} and implies that the dual Steenrod algebra is free over the cohomology of a point and hence flat, as in the equivariant case.
Thus the methods employed in \Cref{sec:power ops} can be applied directly to show that the algebraic condition is satisfied in the motivic setting.
This is also relevant for the results in \Cref{sec:Hinfty}.

Much has been done by Dugger and Isaksen with the affect of the Betti realization map on the relevant Adams spectral sequences.
The papers \cite{DuggerIsaksenEquivandReal} and \cite{DuggerIsaksenLowMWstems} are very accessible computational accounts of what can be done.

\subsection{Motivic extended powers of spheres}
\label{subsec: motivic extended powers of spheres}
In this subsection, we compute the attaching maps of naive extended powers of spheres in the category of motivic spectra over $Spec(\R)$.
This is possible due to the above described nice properties of Betti realization and our work in both \Cref{sec:Thom} and \Cref{sec:rppqs}.
The fact that Betti realization behaves well with respect to cofiber sequences and is monoidal is the main tool that we use, thus the work of Heller and Ormsby cited above is crucial.

\begin{defn}
The motivic naive extended power functor is defined as 
\[
D_n^{tr}(X):=E\Sigma_{n+}\wedge_{\Sigma_n}X^{\wedge n}
\]
where $E\Sigma_{n+}$ is the suspension spectrum of the motivic space that assigns the simplicial set $E\Sigma_{n+}$ to every smooth scheme over $\R$.
The functor $D_n^{tr}$ has a natural filtration induced by the skeletal filtration of $E\Sigma_{n+}$.
\end{defn}
We do not attach notation to the motivic extended powers distinguishing it from the equivariant one.
This is because in practice they will always be applied to spheres and when the distinction between the two cases is relevant the spheres will be decorated appropriately.
As the Betti realization functor of Heller and Ormsby is a symmetric monoidal left Quillen functor we have the following result.

\begin{lemma}
When applied to cellular motivic spectra, the Betti realization functor commutes with the naive extended power construction.
That is, we have a weak equivalence
\[
Re^{C_2}_B(D_n^{tr(k)}(X))\we D_n^{tr(k)}(Re^{C_2}_B(X))
\]
where the superscript $(k)$ indicates that we are only taking the $k$-skeleton of $E\Sigma_n$.
\end{lemma}
At this point, the reader is owed an apology for the overloaded notation.
We know of no way to avoid this and avoid ambiguity between the two different settings.

In particular this gives us that $Re^{C_2}_B(D_2^{tr}(S^{p,q}_{\R}))\we D_2^{tr}(S^{p,q}_{C_2})$ which we happen to know the attaching maps of.
The (projections of the) attaching maps are computed as elements of $\pi_{0,0}(S^{0,0}_{C_2})\iso \Z^2$ by \Cref{prop:attaching maps}.

\begin{lemma}
\label{lemma:betti on pi_0}
The Betti realization functor $Re^{C_2}_B$ induces an isomorphism
\[
\pi_{0,0}(S^{0,0}_{\R}) \lra \pi_{0,0}(S^{0,0}_{C_2})
\]
\end{lemma}

This is an easy consequence of the section $c^*_{\C/\R}$ that Heller and Ormsby produce in \cite{HellerOrmsby}.
Note that $c^*_{\C/\R}(S^{0,0}_{C_2})=S^{0,0}_{\R}$.
As $c^*_{\C/\R}$ is a section of $Re^{C_2}_B$ the composite $Re^{C_2}_B\circ c^*_{\C/\R}$ induces the identity map on $\pi_{0,0}(S^{0,0}_{C_2})$ and so Betti realization must induce an isomorphism.
This follows from Morel's computation that $\pi_{0,0}(S^{0,0}_{\R})\iso \Z \dsum \Z$.
We learned this argument from Jeremiah Heller.
While the results of \cite{HellerOrmsby} only hold up to completion issues this restriction has been removed by their subsequent work in \cite{HellerOrmsbywithoutcompletion}.

\begin{cor}
\label{cor:motivic attaching maps}
As motivic cell complexes over $\R$
\[
D_2^{tr (k)}S^{n,m}\iso D_2^{tr (k-1)}S^{n,m} \cup_{f^1_{2n+k-1,2m}} CS^{2n+k-1,2m}
\]
where the attaching map of the $(i,j)$-cell, is given by
\[
\widetilde{f}_{i,j} = \left\{
\begin{array}{ll}
1-\varepsilon & i \equiv 1, j \equiv 1 \pmod{2} \\
2  & i \equiv 1, j \equiv 0 \pmod{2} \\
1+\varepsilon & i \equiv 0, j \equiv 1 \pmod{2} \\
0 & i \equiv 0, j \equiv 0 \pmod{2} \\
\end{array}
\right\}
\]
in $\pi_{0,0}(S^{0,0}_{\R})$.
\end{cor}
While in the equivariant case we have two different descriptions, we only actually require this one for the proof of our main theorem.
This is because we are using the naive extended power construction and so we never add twisted cells.
\begin{proof}
Here we compute the attaching maps of the motivic extended powers by examining the Betti realization of the cellular filtration:

\[
\xymatrix{
S^{2p,2q}_{C_2}\we Re^{C_2}_B(D_2^{tr (0)}(S^{p,q}_{\R})) \ar[d] \ar@{^{(}->}[r]&
Re^{C_2}_B(D_2^{tr (1)}(S^{p,q}_{\R})) \ar[d] \ar@{^{(}->}[r]&
Re^{C_2}_B(D_2^{tr (2)}(S^{p,q}_{\R})) \ar[d] \ar@{^{(}->}[r]&
\cdots\\
S^{2p,2q}_{C_2}\we D_2^{tr(0)}(S^{p,q}_{C_2}) \ar[d] \ar@{^{(}->}[r]&
D_2^{tr(1)}(S^{p,q}_{C_2}) \ar[d] \ar@{^{(}->}[r]&
D_2^{tr(2)}(S^{p,q}_{C_2}) \ar[d] \ar@{^{(}->}[r]&
\cdots \\
S^{2p,2q}_{C_2}\we \RP^{2p,2q}/\RP^{2p-1,2q} \ar@{^{(}->}[r]&
\RP^{2p+1,2q}/\RP^{2p-1,2q} \ar@{^{(}->}[r]&
\RP^{2p+2,2q}/\RP^{2p-1,2q} \ar@{^{(}->}[r]&
\cdots.
}
\]
The vertical arrows from the first to the second row are gotten by the above property of Betti realization.
The vertical maps from the second to third row are the homeomorphisms of \Cref{cor:equiv extended powers}.
The attaching maps we wish to compute are the various composites
\[
S^{2p+n,2q}_{\R}\lra D_2^{tr (n)}(S^{p,q}_{\R}) \lra S^{2p+n,2q}_{\R}
\]
as elements of $\pi_{0,0}(S^{0,0}_{\R})$.
We know that the cofiber of 
\[
D_2^{tr (n-1)}(S^{p,q}_{\R}) \lra D_2^{tr (n)}(S^{p,q}_{\R})
\]
is $S^{2p+n,2q}$ via part two of the Geometric condition, see \Cref{defn:geometric condition}.
However, Betti realization being full and faithful implies that
\[
Re_B^{C_2}:\pi_{0,0}(S^{0,0}_{\R}) \lra \pi_{0,0}(S^{0,0}_{C_2})
\]
induces an isomorphism.
Thus the attaching map follows from \Cref{prop:attaching maps}.
\end{proof}

\subsection{The Geometric and Homotopical conditions in the motivic setting}
\label{subsec:geometric and homotopical condition motivic}
While the Geometric condition is a consequence of model categorical facts, see \Cref{prop: geometric condition}, establishing that the Homotopical condition holds requires a little bit more care.
However, both can be established just as in the equivariant case, which is the goal of this short section.
The arguments are exactly those of \Cref{sec:Hinfty} and so we don't repeat them here but we will provide references for the necessary facts.

\begin{prop}
\label{prop:geometric and homotopical condition motivic}
\begin{itemize}
\item The naive extended power construction $\Ei\Sigma_{n+}\wedge_{\Sigma_n}(-)^{\wedge n}$ satisfies the geometric condition.
\item The motivic Adams spectral sequence over $Spec(\R)$ based on motivic cohomology, as developed by Dugger and Isaksen in \cite{DuggerIsaksenAdams}, satisfies the homotopy condition.
\end{itemize}
\end{prop}

The first part follows the exact same line of argument as in the equivariant setting.
It is not clear to us that the more nuanced extended power constructions, the more geometrically relevant ones, will satisfy this condition as we are unfamiliar with them.

The second part also follows from the exact same arguments as in \Cref{sec:Hinfty}.
There, we only needed to know the existence of certain spectral sequences and that the homology of various spectra was free over the coefficients.
The motivic UCSS and KSS are developed in \cite{DuggerIsaksenMotivicCell} as Proposition 7.7.
The freeness of the dual Steenrod algebra, see \Cref{subsec: motivic algebraic condition}, and the motivic homology of extended powers of spheres ensure that these spectral sequences collapse just as in the equivariant situation.
Also, the obstruction theoretic result of \cite{ChristensenDwyerIsaksenObstModelCats} also applies in this context.
It is worth recalling that all of the above takes place in the subcategory of cellular motivic spectra where weak equivalences are detected by isomorphisms on bigraded homotopy groups, just as was the case in our equivariant situation.

\section{$d_2$ in Adams spectral sequence}
\label{sec:equi diff}
Now we arrive at our main theorem.
Our argument will be given ambiguously, this means that all indications as to whether or not we are using equivariant or motivic objects will be suppressed.
We will now combine the results of Sections $5-8$, as well as those of \Cref{sec:motivic}, to prove our main theorem, following the strategy outlined in \Cref{sec:outline}. 
Frequently, we will not distinguish between the class in the spectral sequence, its representative in the $E_2$-page, or the map of filtered spectra into the Adams tower representing it.
Note that by Adams tower we mean what many refer to as Adams resolution, for example, as in chapter IV of \cite{HRS}. 
After our main proof we will give sample applications in both the equivariant and motivic cases.

\subsection{Proof of the main theorem}
\label{subsec:proof}
We wish to compute differentials on classes in a homotopy exact couple.
Recall from \Cref{defn:Hinfty filtration} that an $\Hi$-structure on $Y_{\bu}$ gives us a map of filtrations 
\[
D^{tr(n)}_2\Gamma_m(\Yb):=S(\R^{n,0})_+\wedge_{\Sigma_2} \Gamma_m \to Y_{m+n}.
\]
This restricts to a map of the form 
\[
D^{tr(n)}_2(Y_k)\hra D^{tr(n)}_2 \Gamma_{2k}(Y_{\bu}) \sr{\xi_2}\to Y_{2k+n}.
\]
Our grading convention differs from that of Bruner.
Our Adams tower occurs in negative degrees while our cellular filtration of $E\Sigma_2^{(\bu)}$ is in positive degrees.
Bruner's Adams tower is in positive degrees and $E\Sigma_2^{(\bu)}$ is in negative degrees in \cite{HRS}.

The classes in a spectral sequence are represented by maps into the Adams tower from filtered spectra of the form:
\[
U(r,s,(p,q)): \cdots \lra * \lra S^{p-1,q} \lra S^{p-1,q} \lra \cdots \lra S^{p-1,q} \lra CS^{p-1,q} \lra CS^{p-1,q} \lra \cdots.
\]
We briefly recall the theory of universal examples for exact couples as discussed in \Cref{sec:outline}.
The spectral sequence associated with this filtration has only two nonzero elements, one representing the class $x$ which is supported on the leftmost $CS^{p-1,q}$ and $d_r(x)$ which is supported on the leftmost $S^{p-1,q}$. 
Consider the filtered spectrum $U(r,s,(p,q))$
\[
\xymatrix{
\ast \ar[r]&
S^{p-1,q} \ar[r] \ar[dl]&
S^{p-1,q}\ar[rr] \ar[dl]&
&
S^{p-1,q}\ar[r] \ar[dl]&
CS^{p-1,q}\ar[r] \ar[dl]&
CS^{p-1,q} \ar[dl]\\
\ast&
S^{p-1,q}&
\cdots&
\ast&
S^{p,q} \ar@/^1pc/[lll]^{d_r'}&
\ast
}
\]
where we write $d_r'$ because there is a degree shift and the actual differential has codomain $S^{p,q}$.
The filtration $U(r,s,(p,q))$ maps into the Adams filtration
\[
\xymatrix{
U(r,s,(p,q)) \ar[d]&
&
\ast \ar[r] \ar[d]&
S^{p-1,q} \ar[r] \ar[d]&
S^{p-1,q}\ar[rr] \ar[d]&
&
S^{p-1,q}\ar[r] \ar[d]&
CS^{p-1,q}\ar[r] \ar[d]&
CS^{p-1,q} \ar[d]\\
\Yb&
&
Y_{s-r-1} \ar[r]&
Y_{s-r}\ar[r]&
\cdots\ar[rr]&
&
Y_{s-1}\ar[r]&
Y_s\ar[r]&
Y_{s+1}}
\]
where it detects $d_r$ on a class of Adams filtration $s$ and geometric dimension $(p,q)$.

\begin{thm}
\label{thm:diff}
In the $C_2$-equivariant or real motivic Adams spectral sequence converging to the bigraded stable homotopy groups of spheres, for a permanent cycle $x \in \Ext^{s,p,q}_{\cA_*}(\M_2,\M_2)$, we have 
\[
d_2 sq^{i-1} x =
\alpha_{i,q} sq^i x
\]
where 
\[
\alpha_{i, q} = \left\{
\begin{array}{ll}
h_0  & i \equiv 1, q \equiv 1 \pmod{2} \\
h_0 + \rho h_1 & i \equiv 1, q \equiv 0 \pmod{2} \\
\rho h_1 & i \equiv 0, q \equiv 1 \pmod{2} \\
0 & i \equiv 0, q \equiv 0 \pmod{2} \\
\end{array}
\right.
\]
in
$\Ext^{1,1,0}_{\mathcal{A}}(\M_2,\M_2)$. 
\end{thm}

Recall that $\M_2$ are the coefficients of the homology theory, see \Cref{subsec:recollections equivariant coh} or \Cref{subsec: motivic algebraic condition}
The trigrading on this $\Ext$-group is the expected one, $s$ is the homological degree which corresponds to the Adams filtration while $p$ and $q$ come from the bigrading in the underlying category of $\cA_*$-comodules.
Recall that by definition the classes $h_0\in \Ext^{1,1,0}_{\cA_*}(\M_2,\M_2)$ and $h_1\in \Ext^{1,2,1}_{\cA_*}(\M_2,\M_2)$ are induced by $sq^1\in\mc{A}$ and $sq^2\in\mc{A}$ respectively.

\begin{proof}
We wish to look at what happens when we take the extended power construction and apply it to the map representing a permanent cycle.
A map of filtered spectra of the form
\[
\xymatrix{
U(\infty,s,(p,q)) \ar[d]&
&
\cdots\ar[r] \ar[d]&
\ast \ar[r] \ar[d]&
S^{p,q} \ar[r] \ar[d]&
S^{p,q} \ar[r] \ar[d]&
\cdots \ar[d]
\\
\Yb&
&
\cdots\ar[r]&
Y_{s-1} \ar[r]&
Y_s \ar[r]&
Y_{s+1} \ar[r]&
\cdots  
}
\]
detects a permanent cycle of Adams filtration $s$ and geometric dimension $(p,q)$.

It is to these universal examples that we will apply the extended power construction.
We will then examine their associated spectral sequence coming from the cellular filtration in order to obtain formulas for $d_2$ of power operations applied to permanent cycles.
Applying the extended power construction of a filtration given as \Cref{defn:Hinfty filtration} to a permanent cycle gives the filtration
\[
\cdots \lra \ast \lra D^{tr(0)}_2 S^{p,q} \lra D^{tr(1)}_2 S^{p,q} \lra D^{tr(2)}_2 S^{p,q} \lra \cdots
\]
which maps into the extended power construction applied to the Adams filtration
\[
\cdots \lra \tG^2_{2s-1}(\Yb) \lra \tG^2_{2s}(\Yb) \lra \tG^2_{2s+1}(\Yb) \lra \cdots
\]
by functoriality, where 
\[
\tG^2_s(\Yb):=\bigcup_{n+k=s}D^{tr(n)}_2 \Gamma_k(Y_{\bu})
\]
as in \Cref{sec:outline}.
Composing the induced map between the quadratic constructions and the $\Hi$-structure of the the Adams tower $Y_{\bu}$ provided by \Cref{thm: hinfty}, we obtain the map of filtrations
\[
\xymatrix{
\ast \ar[r] \ar[d]&
D^{tr(0)}_2 S^{p,q} \ar[r] \ar[d]&
D^{tr(1)}_2 S^{p,q} \ar[r]\ar[d]&
D^{tr(2)}_2 S^{p,q} \ar[d]
\\
\tG^2_{2s-1}(\Yb) \ar[r] \ar[d]&
\tG^2_{2s}(\Yb) \ar[r] \ar[d]&
\tG^2_{2s+1}(\Yb) \ar[r] \ar[d]&
\tG^2_{2s+2}(\Yb)\ar[d]
\\
Y_{2s-1} \ar[r]&
Y_{2s}\ar[r]&
Y_{2s+1}\ar[r]&
Y_{2s+2}.
}
\]

When the above construction is applied to a permanent cycle $x$, we get that the classes 
\[
\left \{ sq^{\vert x\vert}(x), sq^{\vert x\vert-1}(x), \ldots, sq^2(x),sq^1(x),sq^0(x)\right \}
\]
are supported on the image of the above diagram.
Specifically, $sq^i(x)$ is supported on the $(p-i)$-cell of $D^{tr}_2S^{p,q}$.
This follows from a combination of the Geometric Condition and the definition of the operations constructed in \Cref{sec:power ops}.
Using the Geometric Condition we know that the homotopy of the associated graded complex of the filtration
\[
\tG^2_{\bu}(\Yb):=\bigcup_{n+k=\bu} D^{tr(n)}_2 \Gamma_k(Y_{\bu})
\]
is the complex with $s$th term 
\[
\bigoplus_{n+k=s} C_n(E\Sigma_2) \tensor \pi_*(E^0(\Yb))^{\tensor 2}_k.
\]
This complex maps to the associated graded complex of the Adams tower, and this map is precisely the structure map constructed in \Cref{sec:power ops} which induces power operations in $\Ext$.

For example, since $x$ is supported on $Y_s/Y_{s-1}$ and of geometric dimension $(p,q)$, $sq^s(x)=x^2$ is supported on $Y_{2s}/Y_{2s-1}$ and of geometric dimension $(2p,2q)$.
This can also be represented by the composite map 
\[
S^{2p,2q} \sr{1}\lra S^{p,q} \wedge S^{p,q} \sr{x \wedge x} \lra Y_s \wedge Y_s \sr{\mu} \lra Y_{2s}
\]
This composite is the restriction of the structure map $\xi$ to the $0$-cell of the quadratic construction by assumption.
We can perform similar lifts of the classes $sq^i(x)$.

For the moment we will focus on the $C_2$-equivariant setting.
To determine differentials, we consider the spectral sequence associated with the cellular filtration of $D^{tr}_2 S^{p,q}_{C_2}$.
In the $C_2$-equivariant setting, we computed the cellular filtration of these constructions in \Cref{sec:Thom}.
We found that 
\[
D^{tr(p-i)}_2(S^{p,q}_{C_2})\iso \sus^{p,q} \RP^{2p-i, q}/\RP^{p-1,q}.
\]
The attaching maps of the cells are the differentials in the spectral sequence associated with the skeletal filtration of 
\[
\sus^{p,q} \RP^{2p-i, q}/\RP^{p-1,q}.
\]
We computed those attaching maps in \Cref{sec:rppqs}.
This is where we use fourth ingredient that is discussed in \Cref{hypotheses}.

Our filtration 
\[
\cdots \lra \ast \lra D^{tr(0)}_2 S^{p,q}_{C_2} \lra D^{tr(1)}_2 S^{p,q}_{C_2} \lra D^{tr(2)}_2 S^{p,q}_{C_2} \lra \cdots
\]
becomes
\[
\cdots \lra \ast \lra \sus^{p,q} \RP^{p, q}/\RP^{p-1,q} \lra \sus^{p,q} \RP^{p+1, q}/\RP^{p-1,q} \lra \sus^{p,q} \RP^{p+2, q}/\RP^{p-1,q} \lra \cdots.
\]
We will abbreviate this as
\[
\cdots \lra \ast \lra S^{2p,2q}_{C_2} \lra \sus^{p,q} X_1 \lra \sus^{p,q} X_2 \lra \cdots
\]
where each $X_i:=\RP^{p+i, q}/\RP^{p-1,q}$.
To obtain information about $d_2(sq^{i-1}(x))$ we focus on $X_{p-i} \subset X_{p-i+1}$.
The quotient $X_{p-i+1}/X_{p-i}$ is a two cell complex.
As 
\[
X_{p-i+1}/X_{p-i} \we \RP^{2p-i+1,q}/\RP^{2p-i,q},
\]
we can use \Cref{prop:attaching maps} to determine the projection of the attaching map of the cell supporting $sq^{i-1}(x)$.
The projection of the attaching map of the cell in dimension $(2p-i+1,q)$ to the cell in $(2p-i,q)$ is the same as that of same map that attaches the $(2p-i+1,q)$ cell to $\RP^{2p-i,q}$, namely $\widetilde{f}^1_{2p-i,q}$.

In the real motivic setting, we obtain the same result for the projection of the attaching maps.
This is a direct consequence of the results of \Cref{subsec: motivic extended powers of spheres}.
Essentially, we don't have the identification 
\[
D^{tr(p-i)}_2(S^{p,q}_{C_2})\iso \sus^{p,q} \RP^{2p-i, q}/\RP^{p-1,q}.
\]
as we did in the equivariant situation.
However, we compute the projections of the attaching maps by way of the good formal properties of Betti realization.
So we can also identify the differentials in the cellular spectral sequence associated to the filtration
\[
\cdots \lra \ast \lra D^{tr(0)}_2 S^{p,q}_{\R} \lra D^{tr(1)}_2 S^{p,q}_{\R} \lra D^{tr(2)}_2 S^{p,q}_{\R} \lra \cdots
\]
using \Cref{cor:motivic attaching maps}.

The projection of these attaching maps determine the differentials in the associated spectral sequences in both the equivariant and motivic cases
\[
\xymatrix{
\sus^{p,q}X_{p-i-1}\ar[rr]&
&
\sus^{p,q}X_{p-i}\ar[rr] \ar[dl]&
&
\sus^{p,q} X_{p-i+1} \ar[r] \ar[dl]&
\cdots\\
&
S^{3p-i,2q}&
&
S^{3p-i+1,2q} \ar@/^1pc/[ll]^{d_1'}.&
&
}
\]
We look at the universal examples mapping into this spectral sequence and compose it with the $\Hi$-structure maps of the Adams filtration to obtain
\[
\xymatrix{
\ast \ar[r] \ar[d]&
S^{3p-i,2q} \ar[r] \ar[d] \ar@/_2pc/@{-->}[ddd]_{\alpha}&
CS^{3p-i+1,2q} \ar[r]\ar[d]&
\cdots
\\
\sus^{p,q}X_{p-i-1}\ar[r] \ar[d]&
\sus^{p,q}X_{p-i}\ar[r] \ar[d]&
\sus^{p,q} X_{p-i+1} \ar[r] \ar[d]&
\cdots\\
\tG^2_{2s-1}(\Yb) \ar[r] \ar[d]&
\tG^2_{2s}(\Yb) \ar[r] \ar[d]&
\tG^2_{2s+1}(\Yb) \ar[r] \ar[d]&
\cdots
\\
Y_{2s-1} \ar[r]&
Y_{2s}\ar[r]&
Y_{2s+1}\ar[r]&
\cdots.
}
\]
However, the projection of the attaching map may have a higher Adams filtration.
In fact, $2,1+\varepsilon$ and $1-\varepsilon$ all have Adams filtration $1$.
Thus we have the following lift of the boundary map
\[
\xymatrix{
&
S^{3p-i,2q}\ar[d]&
\\
&
S^{0,0}\wedge Y_{2s}\ar[d]\ar@{-->}[dl]_{\alpha}&
\\
Y_{-1}\wedge Y_{2s}\ar[r] \ar[d]^{\mu_{\bu}}&
Y_0\wedge Y_{2s}&
\\
Y_{2s-1}&
}
\]
which gives us our desired differential.
Here, $\alpha$ is the element in Adams filtration $1$ detecting the relevant attaching map.
The map $2$ is detected in the Adams spectral sequence by $h_0 + \rho h_1$, $1+\varepsilon$ by $\rho h_1$, and $1-\varepsilon$ by $h_0$ in $\Ext^{1,1,0}$.
\end{proof}

\subsection{Preliminary Applications}
\label{subsec: apps}
The above formula can be used to compute differentials in the $RO(C_2)$-graded and real motivic Adams spectral sequences.
The classical version of this theorem has the Hopf invariant one theorem as an application.
The theorem predates the work of Bruner, however his results in \cite{HRS} give a nice proof of the result.
For $n\geq 1$ the formula says
\[
d_2(h_{n+1})=h_0h_n^2.
\]
For $n=1,2,3$ the product on the right is $0$ in $\Ext$.
It is nonzero for higher $n$ and thus we obtain the above theorem since the classes $h_n$ do not survive the Adams spectral sequence.
Note that this requires understand what happens to non-permanent cycles, a case which Bruner handles.

We can use what is known, via the isomorphism of Dugger and Isaksen from \cite{DuggerIsaksenEquivandReal} and their computations in \cite{DuggerIsaksenLowMWstems}, to do similar computations in the equivariant and motivic setting.
The Dugger-Isaksen isomorphism is only in the range $p-s\geq 3q-5$.
Thus we will not be able to make complete determinations, but we will be able to do a few computations.

Classically, $sq^0(h_n)=h_{n+1}$ for every $n$, but this only holds for $n\geq 1$ in the equivariant and motivic settings.
Over $Spec(\mathbb{C})$ we have that where $sq^0(h_0)=\tau h_1$.
To apply our formula we require a permanent cycle.
Motivically, $h_1$ is in the $0th$ Milnor-Witt stem and the Adams differentials decrease this degree.
Therefore $h_1$ must be a permanent cycle in the motivic Adams spectral sequence as Morel has shown that the negative Milnor-Witt stems are trivial in \cite{MorelConnectivity}.
In this range we can use the Dugger-Isaksen isomorphism of stable homotopy groups to see that $h_1$ must be a permanent cycle equivariantly as well.

So we begin by considering the equivariant setting.

\begin{cor}
In the $RO(C_2)$-graded and real motivic Adams spectral sequences we have that
\begin{center}
$d_2(h_2)=h_0h_1^2=0$, $d_2(h_3)=(h_0+\rho h_1)h_2^2=0$, and $d_2(h_4)=(h_0+\rho h_1)h_3^2=h_0 h_3^2\neq=0$.
\end{center}
\end{cor}
That the products on the right are $0$ are computations due to Dugger and Isaksen, see the charts in \cite{DuggerIsaksenLowMWstems}.

\begin{proof}
We will be using the isomorphism of Dugger and Isaksen in order to get a hold of the relevant equivariant $\Ext$ groups.
First, we have $d_2(h_2)=h_0 h_1^2$ by the above formula since $h_1\in\Ext^{1,2,1}_{\cA_*}(\M_2,\M_2)$.
However, this product is $0$ in $\Ext$ just as in the classical situation.
Possible higher differentials on $h_2$ are $d_r(h_2)=\rho^{r-1}h_1^{r+1}$, however $h_1h_2=0$ and as $h_1$ is a permanent cycles if $d_r(h_2)=\rho^{r-1}h_1^{r+1}$ then we would have a contradiction.
The triviality of the product and the nonexistence of other possible targets for differentials both use the Dugger-Isaksen isomorphism of $\Ext$'s.
As the isomorphism of $\Ext$'s extends to one of Adams spectral sequences in this range we can deduce that $h_2$ is a permanent cycle motivically as well.

Similarly, $d_2(h_3)=(h_0+\rho h_1)h_2^2=0$ as $h_2\in\Ext^{1,4,2}_{\cA_*}(\M_2,\M_2)$, and the class $h_2^2$ is annihilated by both $h_0$ and $\rho h_1$.
To compute $d_2(h_4)$ we consider the motivic case and use that Betti realization induces a map of spectral sequences.
This is the most we can accomplish without greater knowledge of the relevant $\Ext$ groups.
Knowing what possible targets for differentials are would be very helpful as we can use the isomorphism of Dugger and Isaksen in conjunction with the product structure to rule out possible differentials.

Now we turn to the motivic setting over $Spec(\R)$.
We also have that motivically $d_2(h_3)=(h_0+\rho h_1)h_2^2=0$.
By examination of the charts in \cite{DuggerIsaksenLowMWstems} there are no targets for higher differentials supported on $h_3$, and so it is a permanent cycle.
Thus we have $d_2(h_4)=(h_0+\rho h_1)h_3^2=h_0 h_3^2$ since $h_3\in\Ext^{1,8,4}_{\cA_*}(\M_2,\M_2)$, and $h_1 h_3^2=0$.
To compute these products on the right see the charts at the end of \cite{DuggerIsaksenLowMWstems}.
This product $h_0h_3^2$ is nonzero equivariantly as well as motivically and so $h_4$ is not a permanent cycle.
\end{proof}

In order to compute $d_2(h_5)$, further analysis is needed since $h_5=sq^0(h_4)$ and $h_4$ is not a permanent cycle, our results do not apply.
However, we expect that further work in this direction, following Bruner's work \cite{HRS}, will show that $d_2(h_5)=(h_0 +\rho h_1)h_4^2$.

\addcontentsline{toc}{section}{References}
\bibliography{Bibliography}{}
\bibliographystyle{plain}

\end{document}